\documentclass{article}

\usepackage{microtype}
\usepackage{graphicx}
\usepackage{subfigure}
\usepackage{booktabs} 

\usepackage[hidelinks]{hyperref}

\usepackage[accepted]{icml2023}

\usepackage{amsmath, cases}
\usepackage{amssymb}
\usepackage{mathtools}
\usepackage{amsthm}

\theoremstyle{plain}
\newtheorem{theorem}{Theorem}[section]
\newtheorem{proposition}[theorem]{Proposition}
\newtheorem{lemma}[theorem]{Lemma}

\theoremstyle{definition}

\newtheorem{assumption}[theorem]{Assumption}
\theoremstyle{remark}
\newtheorem{remark}[theorem]{Remark}

\usepackage[textsize=tiny]{todonotes}

\usepackage[linesnumbered, ruled, vlined, noend]{algorithm2e} 
\usepackage[hang,flushmargin]{footmisc}
\usepackage{colortbl}
\usepackage{algpseudocode}
\usepackage{graphicx}
\usepackage{pifont}
\usepackage{subfigure}
\usepackage[normalem]{ulem}
\usepackage[toc,page,header]{appendix}
\usepackage{minitoc}

\usepackage{hyperref}
\usepackage[capitalize,noabbrev]{cleveref}

\newcommand*\OK{\ding{51}}
\newcommand{\NO}{\text{\ding{55}}}

\newcommand{\eg}{\emph{e.g.}}
\newcommand{\ie}{\emph{i.e.}}
\algloop{InParallel}
\algloopdefx{InParallel}[1]{\textbf{in parallel} #1 \textbf{continuously do}}{}

\icmltitlerunning{DADAO: Decoupled Accelerated Decentralized Asynchronous Optimization}

\begin{document}
\doparttoc 
\faketableofcontents 

\twocolumn[
\icmltitle{DADAO: Decoupled Accelerated Decentralized Asynchronous Optimization}



\icmlsetsymbol{equal}{*}

\begin{icmlauthorlist}
\icmlauthor{Adel Nabli}{yyy,comp}
\icmlauthor{Edouard Oyallon}{yyy}
\end{icmlauthorlist}

\icmlaffiliation{yyy}{Sorbonne Université, CNRS, ISIR, Paris, France}
\icmlaffiliation{comp}{Mila, Concordia University, Montréal, Canada}

\icmlcorrespondingauthor{Adel Nabli}{adel.nabli@sorbonne-universite.fr}

\icmlkeywords{Machine Learning, ICML, Convex Optimization, Decentralized Optimization, Asynchronous Optimization}

\vskip 0.3in
]



\printAffiliationsAndNotice{}  

\begin{abstract}
  This work introduces DADAO: the first decentralized, accelerated, asynchronous, primal, first-order algorithm to minimize  a sum of $L$-smooth and $\mu$-strongly convex functions distributed over a given network of size $n$.  Our key insight is based on modeling the local gradient updates and gossip communication procedures with separate independent Poisson Point Processes. This allows us to decouple the computation and communication steps, which can be run in parallel, while making the whole approach completely asynchronous. This leads to communication acceleration compared to synchronous approaches. Our new method employs primal gradients and does not use a multi-consensus inner loop nor other ad-hoc mechanisms  such as Error Feedback, Gradient Tracking, or a Proximal operator. By relating the inverse of the smallest positive eigenvalue of the Laplacian matrix $\chi_1$ and the maximal resistance $\chi_2\leq \chi_1$ of the graph to a sufficient minimal communication rate between the nodes of the network, we show that  our algorithm requires $\mathcal{O}(n\sqrt{\frac{L}{\mu}}\log(\frac{1}{\epsilon}))$ local gradients and only  $\mathcal{O}(n\sqrt{\chi_1\chi_2}\sqrt{\frac{L}{\mu}}\log(\frac{1}{\epsilon}))$ communications to reach a precision $\epsilon$, up to logarithmic terms. Thus, we simultaneously obtain an accelerated rate for both computations and communications, leading to an improvement over state-of-the-art works, our simulations further validating the strength of our relatively unconstrained method.

\end{abstract}
\section{Introduction}In recent years, the increased amount of available data as well as the proliferation of highly-parallelizable and connected hardware have brought significant changes in the way we process data. These developments have led to the need for efficient and scalable methods for distributed optimization, particularly in the context of machine learning. Indeed, in scenarios where data is distributed across multiple nodes, such as in edge computing or distributed sensor networks, leveraging the local resources of each device is a topic of significant interest. In other settings, such as clusters, spreading the compute load is ideally done to obtain a linear speedup in the number of nodes. In a typical distributed training framework, the goal is to minimize a sum of  functions $(f_i)_{i\leq n}$ split across $n$ nodes of a computer network. A corresponding optimization procedure involves alternating local computations on the nodes and communications along the edges $\mathcal E$ of the network. In the decentralized setting, there is no central machine aggregating the information sent by the workers: nodes are only allowed to communicate with their neighbors in the network. This work addresses simultaneously  multiple limitations of existing decentralized algorithms while guaranteeing fast convergence rates.  
\paragraph{Synchronous lock.} Optimal methods~\cite{scaman17a,kovalev2021lower} have been derived for synchronous first-order algorithms, whose executions are  blocked until all nodes have reached a predefined state (\eg, they must all finish computing local gradients before the round of communication begins), which limits their efficiency in practice as they can heavily be impacted by a few slow nodes or edges in the graph (the \textit{straggler problem}). To tackle the synchronous lock, we rely on the continuized framework \citep{even2021continuized}, itself derived from the randomized gossip model \cite{Boyd2006}. In randomized gossip, nodes update their local values at random times using pairwise communication updates named gossip. Thus, iterates are randomized, labeled with a continuous-time index that only need local clocks to be synchronized at the beginning of the procedure (in opposition to a \textit{global} iteration count that has to be known by all at all time) and performed locally with no regards to a specific global ordering of events. While based on discrete events and thus readily implementable, the continuized framework simplifies the analysis by leveraging continuous proof tools. 

\paragraph{Coupled lock.}However, in \cite{even2021continuized}, gradient and gossip operations are coupled: each communication along an edge first requires the computation of the gradients of the two functions locally stored on the corresponding nodes. As more communication steps than gradient computations are necessary to reach an $\epsilon$ precision, even in an optimal framework \citep{kovalev2021lower, scaman17a}, the coupling leads to an overload in terms of gradient steps. Moreover, coupling computations and communications implies they must be performed \textit{sequentially}, decoupling them allows both tasks to be done in \textit{parallel}, allowing an additional speedup. 

To our knowledge, our work is the first primal method to tackle those locks simultaneously while obtaining accelerated rates for both computations and communications. We propose a novel algorithm (DADAO: Decoupled Accelerated Decentralized Asynchronous Optimization) based on a combination of similar formulations  to \cite{kovalev2021lower,even2021decentralized, hendrikx2022principled} in the continuized framework of \cite{even2021continuized}. We study:
\begin{align}
\inf_{x\in \mathbb{R}^d} \sum_{i=1}^n f_i(x)\,,\label{main}
\end{align}
where each $f_i: \mathbb{R}^d \rightarrow \mathbb R$ is a $\mu$-strongly convex and $L$-smooth function computed in one of the $n$ nodes of a network. We derive a  first-order optimization algorithm that only uses primal gradients and relies on Point-wise Poisson Processes  (P.P.P.s~\citep{last2017lectures}) modeling of the communication and gradient occurrences, leading to accelerated communication and computation rates. Furthermore, our communication bounds rely on the maximal resistance of a graph rather than the largest eigenvalue of a Laplacian, leading to an additional acceleration compared to works which rely on synchrony. Our framework is based on a simple fixed-point iteration and kept minimal: it only involves primal computations with an additional momentum term. Thus, we do not add other cumbersome designs such as the Error-Feedback or Forward-Backward used in \cite{kovalev2021lower} (whose adaptation to asynchronous settings is for now unclear). While we do not consider the delays bound to appear in practice (we assume instantaneous communications and computations), we remove the coupling lock by performing gradient and gossip steps in parallel. Tab. \ref{tab:other_methods} compares DADAO with other approaches and shows it is the only work to achieve accelerated rates both in number of communication and gradients.

\paragraph{Contributions.} \textbf{(1)} We propose a primal algorithm with provable guarantees in the context of asynchronous decentralized learning.  \textbf{(2)} This algorithm is the first to reach accelerated rates for both communications and computations while not requiring ad-hoc mechanisms obtained from an inner loop. \textbf{(3)} We propose a simple theoretical framework compared to concurrent works, we show that our rates are better than previous works, and \textbf{(4)} we illustrate this theoretical comparison numerically.
\paragraph{Structure of the paper.} In Sec. \ref{hyp}, we describe our work hypothesis and our model of a decentralized environment, while Sec. \ref{dynamic_opti} describes our dynamic. Sec. \ref{guarantee} states our convergence guarantees and highlights that the communication and computational rates of our method are better compared to its competitors. Next, Sec. \ref{algo} explains our implementation of this algorithm, and finally, Sec. \ref{num} verifies our claims numerically. All our experiments are reproducible, using PyTorch \citep{Pytorch2019}, our code being online \footnote{\href{https://github.com/AdelNabli/DADAO/}{https://github.com/AdelNabli/DADAO/}}.
\begin{table*}[!ht]
    \centering
    \caption{This table shows the strength of DADAO compared to concurrent works for obtaining $\epsilon$-precision. $n$ is the number of nodes, $|\mathcal{E}|$ the number of edges, $\frac 1{\chi_1}$ the smallest  positive eigenvalue of a fixed weighted Laplacian $\mathcal{L}$, $\rho$ the eigengap and  $\chi_2\leq \chi_1$ the effective resistance. Note that under reasonable assumptions $\sqrt{\chi_1\chi_2}n =\mathcal{O}(|\mathcal{E}|\sqrt{\rho})$ (see Lemma \ref{chi-lemma}). Async., Comm., Grad., M.-C. and Prox. stand respectively for Asynchrony, Communication steps, Gradient steps., Multi-consensus and Proximal operator. As suggested in their respective papers, all the algorithms are run with $\frac{1}{\Vert \mathcal{L}\Vert}\mathcal{L}$. For OGT, we note that the matrix is stochastic, a more precise comparison is given by Prop. \ref{optimality-dadao}.} 
    \label{tab:other_methods}
\resizebox{2.1\columnwidth}{0.2\linewidth}{
\begin{tabular}{l|ccccccc}
Method & Async.&Decoupled&No Inner Loop&Primal& Total &Total\\
&&&(M.-C. or Prox.)& Oracle &\# Comm.&\# Grad.\\
\hline
MSDA \citep{scaman17a} & $\NO$ &$\NO$&$\NO$&$\NO$&$\sqrt{\rho}|\mathcal{E}|\sqrt{\frac{L}\mu}\log \frac 1 \epsilon $&\cellcolor{green!25}$n\sqrt{\frac L\mu}\log \frac 1 \epsilon $\\
DVR  \cite{hendrikx2020dual}&$\NO$&$\NO$&$\NO$&\OK&$\sqrt{\rho}|\mathcal{E}|\sqrt{\frac{L}\mu}\log \frac 1 \epsilon $&\cellcolor{green!25}$n\sqrt{\frac L\mu}\log \frac 1 \epsilon $\\
ADOM+ \citep{kovalev2021lower}& $\NO$  &$\NO$&$\NO$&\OK&$\rho|\mathcal{E}|\sqrt{\frac{L}\mu}\log \frac 1 \epsilon $&\cellcolor{green!25}$n\sqrt{\frac{L}\mu}\log \frac 1 \epsilon $\\
TVR \citep{hendrikx2022principled} & \NO & \OK & $\NO$ & \OK & $\rho|\mathcal{E}|\frac{L}{\mu}\log \frac 1 \epsilon $&$n\frac{L}\mu\log \frac 1 \epsilon $ \\
AGT \citep{li2021accelerated}&$\NO$&$\NO$&$\NO$&\OK&$\sqrt{\rho}|\mathcal{E}|\sqrt{\frac L\mu}\log \frac 1 \epsilon $&\cellcolor{green!25}$n\sqrt{\frac L\mu}\log \frac 1 \epsilon $\\
OGT \cite{song2021optimal}&\NO&\OK&\OK&\OK&$\sqrt{\rho}|\mathcal{E}|\sqrt{\frac{L}\mu}\log \frac 1 \epsilon $&\cellcolor{green!25}$n\sqrt{\frac L\mu}\log \frac 1 \epsilon $\\
 ESDACD \cite{hendrikx2019accelerated}&\OK&\NO&\OK&\NO&\cellcolor{green!25}$\sqrt{\chi_1\chi_2}n\sqrt{\frac{L}\mu}\log \frac 1 \epsilon $&$\sqrt{\chi_1\chi_2}n\sqrt{\frac{L}\mu}\log \frac 1 \epsilon $\\
 Continuized \citep{even2021continuized}& \OK&$\NO$&\OK&$\NO$&\cellcolor{green!25}$\sqrt{\chi_1\chi_2}n\sqrt{\frac{L}\mu}\log \frac 1 \epsilon $&$\sqrt{\chi_1\chi_2}n\sqrt{\frac{L}\mu}\log \frac 1 \epsilon $& \\
  
 DADAO (ours) & \OK  &\OK&\OK&\OK&\cellcolor{green!25}$\sqrt{\chi_1\chi_2}n\sqrt{\frac{L}\mu}\log \frac 1 \epsilon $&\cellcolor{green!25}$n\sqrt{\frac{L}\mu}\log \frac 1 \epsilon $
\end{tabular} }
\end{table*}

\paragraph{Notations:}  $f=\mathcal{O}(g)$ means there is a constant $C>0$ such that $|f|\leq C|g|$, $\{e_i\}_{i\leq d}$ is the canonical basis of $\mathbb{R}^d,d\in \mathbb{N}$, $\mathbf{1}$ is the vector of 1, $\mathbf{I}$ the identity, $A^+$ is the pseudo-inverse of $A$. We further write $\mathbf{e}_i\triangleq e_i\otimes \mathbf{I}$.

\section{Related Work}\paragraph{Continuized and asynchronous algorithms.} We highly rely on the elegant continuized framework \citep{even2021continuized}, which allows obtaining simpler proofs and brings the flexibility of asynchronous algorithms. We reemphasize that identically to \citep{even2021continuized}, the result of this paper is \emph{a stochastic discrete algorithm with a continuous proof:} our proof framework is not based on the discretization of an Ordinary Differential Equation (ODE) but rather studies \emph{the evolution of a Stochastic Differential Equation (SDE) with jumps}. However, by contrast to \citep{even2021continuized}, in our work, we significantly reduce the necessary amount of gradient steps compared to \cite{even2021continuized} while maintaining the same amount of activated edges. Another type of asynchronous algorithm can also be found in \cite{latz2021analysis}, yet it fails to obtain Nesterov's accelerated rates for lack of momentum. We note that \cite{leblond2018improved} studies the robustness to delays yet requires a shared memory and thus applies to a different context than decentralized optimization. \cite{hendrikx2022principled} is a promising approach for modeling random communication on graphs yet fails to obtain acceleration in a neat framework without inner loops.

\paragraph{Decentralized algorithms with fixed topology.} \cite{scaman17a} is  the first work to derive an accelerated algorithm for decentralized optimization, and it links the convergence speed to the Laplacian eigengap. The corresponding algorithm uses a dual formulation and a Chebychev acceleration (synchronous and only for fixed topology). Yet, as stated in Tab. \ref{complexity-graph}, it still requires many edges to be activated. Furthermore, under a relatively flexible condition on the intensity of our P.P.P.s, we show that our work improves over bounds that depend on the spectral gap. An emerging line of work following this formulation employs the continuized framework \citep{even2020,even2021continuized,even2021decentralized}, but unfortunately do not use a primal oracle, as they rely on the gradients of the Fenchel conjugate. Finally, we note that the work of \cite{even2021decentralized} incorporates delays in their model, showing that, with some adaptation, the continuized methods in \citep{even2021continuized} still converge at a linear rate. Yet transferring this robustness to our setting remains unclear. Reducing the number of communication has been studied in \cite{mishchenko2022proxskip}, but without obtaining accelerated rates. \cite{ADFSSiam2021} allows for fast communication and gossip rates yet requires a proximal step and synchrony between nodes to apply a momentum variable.

\paragraph{Finite sum acceleration.} If each local function $f_i$ is a sum of elementary functions $\sum_{j=1}^mf_{i,j}$ with a favorable conditionning, an additional acceleration is possible, as observed by \cite{ADFSSiam2021,hendrikx2020dual,hendrikx2022principled}, which is a different problem from Eq. \ref{main}. This can be viewed as a cluster of nodes with infinite connectivity and an efficient decentralized algorithm should automatically adapt to such structure. Thus, we focused on the setting $m=1$, which allows a fair comparison with these works.

\paragraph{Error feedback/Gradient tracking.} A major lock for asynchrony is the use of Gradient Tracking (GT)~\citep{koloskova2021improved,nedic2017achieving,li2021accelerated} or Error Feedback~\citep{stich2019error,kovalev2021accelerated}. Indeed, gradient operations are locally tracked by a running-mean variable which must be synchronously updated at each gradient update (in GT, the sum of this distributed variable keeps track of the gradient of the objective function), making it incompatible with an asynchronous framework. \citep{zhang2019fully} uses a GT-\textit{like} procedure, which do not verify this property, at the cost of using a buffer (increasing memory requirements) and worse convergence rates (\eg, not accelerated). Furthermore, acceleration requires an undesirable multi-consensus inner loop. We emphasize that \cite{song2021optimal} allows to decouple the gradient updates from communication, yet the framework is still synchronous, leading to synchronous communication rates, that asynchrony can improve (see Tab. \ref{tab:other_methods}).

\paragraph{Decoupling procedures.}Decoupling subsequent steps of optimization procedures traditionally leads to speed-ups~\citep{ADFSSiam2021,hendrikx2022principled,pmlr-v119-belilovsky20a,belilovsky2021decoupled}. This contrasts with methods which couple gradient and gossip updates, so that they happen in a predefined order, i.e., simultaneously~\citep{even2021continuized} or sequentially~\citep{kovalev2021lower,koloskova2020unified}. In decoupled optimization procedures, inner-loops are not desirable as they require an external procedure that can be potentially slow and need a block-barrier instruction during the algorithm's execution (e.g., \cite{ADFSSiam2021}). It means in particular that it is preferrable to avoid approaches such as Catalyst~\cite{lin2015universal}, multi-consensus steps~\cite{kovalev2021adom} or Tchebychev acceleration of consensus~\cite{scaman17a}.

\paragraph{Resistance of a graph.} The maximal resistance of a graph is a widely studied quantity, particularly in physics~\cite{klein1993resistance,vos2016methods,klein2002resistance}, as it is a refined geometric invariant of graphs. The resistance of a graph corresponds to the commute time of a Markov Chain~\cite{chandra1996electrical}. However, beyond the Continuized framework \cite{even2021continuized} or acceleration of consensus problems \cite{can2022randomized,8308701,BoydResistance}, we are unaware of generic, asynchronous, accelerated decentralized optimization procedures that rely on this quantity. Also, it has a more physical interpretation than the Laplacian's norm, as it can be computed via Ohm's and Kirchhoff's Circuit Laws (see ~\cite{chandra1996electrical}). 

\section{Accelerated Asynchronous Algorithm}
\subsection{Gossip Framework}\label{hyp}
We consider the problem defined by Eq. \ref{main} in a distributed environment constituted by $n$ nodes whose dynamic is indexed by a continuous time index $t\in\mathbb{R}^+$. Each node has a local memory and can compute a local gradient $\nabla f_i$, as well as elementary operations, in an instantaneous manner. As said above, having no delay is less realistic, yet adding them also leads to significantly more difficult proofs whose adaptation to our framework remains largely unclear. Next, we will assume that our computations and gossip result from independent homogeneous P.P.P. with no delay. For the sake of simplicity, we assume that all nodes can compute a gradient at the same rate:
\begin{assumption}[Homogeneous gradient computations]\label{ass1}
The gradient computations are normalized to fire independently at a rate of 1 computation per time unit. For the $i$-th worker, we write $N_i(t)$ the corresponding P.P.P. of rate $1$, as well as $\mathbf{N}(t)=(N_i(t))_{i\leq n}$.
\end{assumption}
\begin{remark}
    The P.P.P $N_i(t)$ on node $i$ means that taking gradient steps at $i$ are discrete events, but the time intervals between two events is a random variable following an exponential law of parameter $1$. Thus, the expected waiting time between two gradient steps on the $i$-th worker is $1$ time unit.
\end{remark}
Next, we model the bandwidth of each connection.  For an edge $(i,j)$ belonging to $\in \mathcal{E}$, the set of edges of a graph we assume connected \eqref{ass2}, we write  $M_{ij}(t)$ the P.P.P. with rate $0 < \lambda_{ij}< \infty$. When this P.P.P. fires, both nodes share and update their local memories. The rate $\lambda_{ij}$ is adjustable locally by  machine $i$ while $\lambda_{ji}$ is controlled by machine $j$. While $\lambda_{ij}$ and $\lambda_{ji}$ may be different, we highlight that the communication process is symmetric, $\ie$~both nodes update their local memories when either one of the two corresponding P.P.Ps fires. Thus, in the corresponding undirected graph $\bar{\mathcal{E}}$, each edge $(i,j)$ will fire at a rate of $\lambda_{ij} + \lambda_{ji}$. Given our notations, if $(i,j)\not\in \mathcal{E}$, then the connection between $(i,j)$ can be thought as a P.P.P. with intensity 0. Taking the $\lambda_{ij}$ as edge weights, we introduce the subsequent graph Laplacian, which is the  expected Laplacian of our graph:
$$\Lambda \triangleq\sum_{(i,j)\in \mathcal{E}}\lambda_{ij}(e_i-e_j)(e_i-e_j)^\mathsf{T}\,.$$ We write $\mathbf{\Lambda}\triangleq\sum_{(i,j)\in \mathcal{E} }\lambda_{ij}(\mathbf{e}_i-\mathbf{e}_j)(\mathbf{e}_i-\mathbf{e}_j)^\mathsf{T}$ its tensorized counter-part that will be useful for our Lyapunov-based proofs. Following \cite{scaman17a}, we will compare this quantity to the following projector:
$$\pi\triangleq\mathbf{I}-\frac 1n\mathbf{1}\mathbf{1}^\mathsf{T}=\frac 1{2n}\sum_{1 \leq i,j \leq n}(\mathbf{e}_i-\mathbf{e}_j)(\mathbf{e}_i-\mathbf{e}_j)^\mathsf{T}\,.$$

In this context, a natural quantity is the algebraic connectivity of our network given by \cite{kovalev2021lower}:
$$\chi_1\triangleq \sup_{\Vert x\Vert=1, x\perp \mathbf{1}} \frac{1}{x^\mathsf{T}\Lambda x}\,.$$

We might also write $\chi_1[\Lambda]$ to avoid confusion, depending on the context. 

Next, the maximal effective resistance of the network, as in \cite{even2021continuized,ellens2011effective}, is:
$$\chi_2\triangleq\frac {1}2\sup_{(i,j)\in \mathcal{E} }(e_i-e_j)^\mathsf{T}\Lambda^{+}(e_i-e_j)\,.$$

A standard quantity \cite{scaman17a}, which is used to control the number of synchronous gossips steps, is the spectral gap, given by:
$$\rho\triangleq \Vert \Lambda\Vert \chi_1\,.$$

We also introduce the value $\kappa$, the ratio of communication frequency between the fastest and slowest edges:
$$\kappa\triangleq \frac{\sup_{(i,j)\in \mathcal{E}}\lambda_{ij}+\lambda_{ji}}{\inf_{(i,j)\in \mathcal{E}}\lambda_{ij}+\lambda_{ji}}\,.$$

This ratio is typically bounded, for instance, in the case of a graph with constant edge weights or for $\lambda_{ij}=\frac 1{d_i}$ with $d_i$ the degree of the $i$-th node in a bounded degree graph or a regular graph.

We prove the following Lemma (proved in Appendix \ref{proof-chi-lemma}), which is useful to control $\chi_1,\chi_2$ and compare our bounds with works that rely on the spectral gap of a graph:
\begin{lemma}[Effective resistance]\label{chi-lemma}The spectrum of $\Lambda$ is non-negative. Also, we have $\chi_1=+\infty$ iff $\bar{\mathcal{E}}$ is not a connected graph. If the graph is connected, then:
$$\frac{n-1}{\mathrm{Tr}\, \Lambda}\leq \chi_2\leq  \chi_1\,.$$
Furthermore, we also have the following:
$$\sqrt{\chi_1\chi_2}\mathrm{Tr}\, \Lambda\leq  \sqrt{\rho}\sqrt{\kappa n |\bar{\mathcal{E}}|}\,.$$
\end{lemma}
The last part of this Lemma indicates that our method requires less communications than methods depending on the spectral gap when no degenerated behavior on the graph's connectivity happens, \ie~ when $\kappa$ is adequately bounded, which is a standard assumption~\cite{NEURIPS2019_Hendrix}. 
\begin{remark}
    For synchronous frameworks \cite{kovalev2021lower, scaman17a}, the spectral quantity extracted from their gossip matrix is a given measure of the connectedness of their graphs that is used afterwards to deduce the right number of synchronous communication rounds between two gradient steps. In our framework, $\Lambda $ directly contains the information of both the topology $\mathcal{E}  $ and the edge communication rates $\lambda_{ij} $, thus $\chi_1[\Lambda]$ must rather be understood as an indicator of how well the graphs connectedness and the chosen communication strategy interact. In fact, we will see later that there is a condition on $\chi_1$, $\chi_2$ for DADAO to converge, see Remark \ref{PhysicalInterpretation} for further discussion.
\end{remark}

\begin{assumption}[Connected graph]\label{ass2}
    The set of edges $\mathcal E$ defines a connected graph such that $\chi_1[\Lambda] <  \infty$.
\end{assumption}

\subsection{Dynamic to optimum}\label{dynamic_opti}
Next, we follow a standard approach~\citep{KovalevNEURIPS2020, kovalev2021lower,salim2022optimal,hendrikx2022principled} for solving Eq. \ref{main} (see Appendix \ref{saddle-point-reform} for details), leading to studying, for $0<\nu<\mu$, the following Lagrangian:
\begin{align*}
\inf_{x \in \mathbb{R}^{n \times d}}\sup_{\substack{y\in \mathbb{R}^{n \times d} \\ z\in \mathbb{R}^{n \times d}}} &\sum_{i=1}^nf_i(x_i) -\frac \nu 2\Vert x\Vert^2- \langle x,y\rangle -\frac 1{2\nu}\Vert \pi z+ y\Vert^2.
\end{align*}
For $f(x)=\sum_{i=1}^n f_i(x_i)$, the saddle points $(x^*,y^*,z^*)$ of the above Lagrangian are given by:
\begin{equation}\label{optimal}
\left\{
    \begin{array}{ll}
\nabla f(x^*)-\nu x^*-y^*&=0\\
y^*+\pi z^*+\nu x^*&=0\\
\pi z^*+\pi y^*&=0\,.\\
\end{array}
\right.
\end{equation}
Our algorithm is based on a fixed-point algorithm to obtain those saddle points, which is a similar idea to \cite{kovalev2021accelerated}, which is only restricted to a setting without communication acceleration. Furthermore, contrary to \cite{kovalev2021lower}, we do not employ  a Forward-Backward algorithm, which requires both an extra-inversion step and additional regularity on the  considered proximal operator.  Not only does this condition not hold in this particular case, but this is not desirable in a continuized framework where iterates are not ordered in a predefined sequence and require a local descent at each instant. Another major difference is that no Error-Feedback is required by our approach, which allows unlocking asynchrony while simplifying the proofs and decreasing the required number of communications. Instead, we show it is enough to incorporate a standard fixed point algorithm, \textit{without any specific preconditioning} (see \cite{condat2019proximal}). We consider the following dynamic:
\begin{align}
dx_t &=\eta(\tilde x_t-x_t)dt-\gamma (\nabla f(x_t)-\nu x_t-    \tilde y_t)\,d\mathbf{N}(t) \nonumber\\
d \tilde x_t&=\tilde\eta( x_t-\tilde x_t)dt-\tilde \gamma(\nabla f(x_t)-\nu x_t-    \tilde y_t)\,d\mathbf{N}(t) \nonumber\\
d \tilde y_t &= -\theta(y_t+z_t+\nu \tilde x_t)dt\nonumber\\
&\quad+  (\delta +\tilde \delta)(\nabla f(x_t) -\nu x_t - \tilde y_t) d\mathbf{N}(t)\nonumber\\
 d y_t &=\alpha (\tilde y_t - y_t)dt  \label{dynamic} \\
 dz_t &=\alpha(\tilde z_t - z_t) dt\nonumber \\
&\quad- \beta\textstyle\sum\limits_{(i,j)\in \mathcal{E} } (\mathbf{e}_i-\mathbf{e}_j)(\mathbf{e}_i-\mathbf{e}_j)^\mathsf{T}(y_t+z_t )dM_{ij}(t)\nonumber\\
 d \tilde z_t &=\tilde \alpha(z_t - \tilde z_t) dt\nonumber\\
&\quad-  \tilde\beta\textstyle\sum\limits_{(i,j)\in \mathcal{E} } (\mathbf{e}_i-\mathbf{e}_j)(\mathbf{e}_i-\mathbf{e}_j)^\mathsf{T} (y_t+z_t)dM_{ij}(t) \, ,\nonumber
\end{align}where $\nu,\tilde\eta,\eta,\gamma,\alpha,\tilde\alpha,\theta,\delta,\tilde\delta,\beta,\tilde\beta$ are undetermined real-valued parameters. As in \cite{nesterov2003introductory}, variables are paired to obtain a Nesterov acceleration. The variables  $(x,y)$ allow decoupling the gossip steps from the gradient steps using independent P.P.P.s. Furthermore, the Lebesgue integrable path of $\tilde y_t$ does not correspond to a standard momentum, as in a continuized framework~\citep{even2021continuized}; however, it turns out to be a crucial component of our method. Compared to \cite{kovalev2021lower}, no extra multi-consensus step needs to be integrated. Our formulation of an asynchronous gossip step is similar to \cite{even2021continuized}, which introduces a stochastic variable on edges;  however, contrary to this work, our gossip and gradient computations are decoupled. We emphasize that while the dynamic \ref{dynamic} is formulated using SDEs~\citep{arnold1974stochastic}, which brings the power of the continuous-time analysis toolbox, it is still \emph{event-based} and thus discrete in nature. Hence, the dynamic can be efficiently implemented in practice as explained in Sec. \ref{algo} and Appendix \ref{PracticalImplem}. 

\subsection{Theoretical guarantees}\label{guarantee}
We prove the following in Appendix \ref{proof-main-thm}.
\begin{theorem}\label{main-thm}
Assume each $f_i$ is $\mu$-strongly convex and $L$-smooth.  Assume \ref{ass1}, \ref{ass2}, and that $2\chi_1\chi_2\leq 1$. Then there exists some parameters for the dynamic Eq. \eqref{dynamic} (given in Appendix \ref{important-lemma}), such that for any initialization $x_0\in \text{ker}(\pi)$, and $\tilde x_0=x_0, y_0=\tilde y_0=\nabla f(x_0)-\frac \mu 2 x_0, z_0=\tilde z_0=-\pi y_0$,  we get for $t\in \mathbb{R}^+$:
\begin{equation*}\mathbb{E}[\Vert x_t-x^*\Vert^2]\leq (\frac 12+\frac {23}8\frac{L}\mu+2\frac {L^2}{\mu^2})\Vert x_0-x^*\Vert^2e^{-\frac{t}{8\sqrt 2}\sqrt{\frac{\mu}L}}\end{equation*}
Also, the expected number of oracle gradient calls is $nt$ and the expected number of edges activated is:
$\frac t 2 \mathrm{Tr}\, \Lambda\,.$\end{theorem}
The condition $2\chi_1\chi_2 \leq 1$ can simply be understood as whether or not the chosen communication strategy suits sufficiently well the graph topology $\mathcal{E} $ for DADAO to converge using the expected rates of communication $(\lambda_{ij})_{(i,j) \in \mathcal E}$ compared to the expected rate of one gradient step per time unit on each node we assumed in Assumption \ref{ass1}. In fact, we make the following interpretation:
\begin{remark}
\label{PhysicalInterpretation}
    Introducing $\lambda \triangleq \sum_{(ij) \in \mathcal E} \lambda_{ij}$ and $\mathcal P_\Lambda \triangleq \frac{2 \Lambda}{\mathrm{Tr}\, \; \Lambda}$, we write $\Lambda$ as the product of these two more interpretable quantities to gain more insight on the condition $2\chi_1[\Lambda] \chi_2[\Lambda] \leq 1$:
\begin{equation}
    \Lambda = \lambda \mathcal P_\Lambda.
\end{equation}
In this setting, $\lambda $ is the expected rate of communication over the \emph{whole graph}, while $\mathcal P_\Lambda$ can be interpreted as the Laplacian of $\mathcal E$ with each edge weighted with its probability of having spiked given a communication happened in the graph. We have:
\begin{equation}
    \chi_1[\Lambda]  = \frac{\chi_1[\mathcal P_\Lambda]}{\lambda} \quad ; \quad \chi_2[\Lambda]  = \frac{\chi_2[\mathcal P_ \Lambda]}{\lambda}. \label{chi_normalized}
\end{equation}
$\mathcal P_\Lambda$ being normalized, we could say that the quantities $\chi_1[\mathcal P_\Lambda], \chi_2[\mathcal P_ \Lambda]$ characterize the graph's connectivity while $\lambda$ is the global rate of communication. Then, using \eqref{chi_normalized}, the condition $2\chi_1[\Lambda] \chi_2[\Lambda] \leq 1$ is equivalent to saying
\begin{equation}
    \sqrt{2 \chi_1[\mathcal P_\Lambda] \chi_2[\mathcal P_ \Lambda]} \leq \lambda,
\end{equation}
meaning that the global communication rate should be larger than some spectral quantity quantifying the graph's connectivity. If structural constraints on the network makes it impossible to verify this condition, as the notion of time is \emph{only} defined through the rate of gradient steps given by Assumption \ref{ass1} ($\lambda$ can be interpreted as \textit{"the expected number of communications in the graph between the expected duration separating two subsequent gradient steps on a given node"}), it only means that the gradient steps are happening too fast compared to the ability of the network to communicate, and the rate of gradient steps should be slow-down by a factor of $\sqrt{2 \chi_1 \chi_2}$.
\end{remark}
Thus, given a graph topology, it is straightforward to obtain a Laplacian verifying $2\chi_1\chi_2 \leq 1$. Indeed, we have the following:
\begin{remark}
\label{laplacian}
    For any graph topology $\mathcal{E}$ and any choice of corresponding graph Laplacian $\mathcal L$ verifying $\chi_1 [\mathcal L]$ bounded, there is a communication strategy $(\lambda_{ij})_{(i,j) \in \mathcal E}$ such that $\Lambda$, the Laplacian of $\mathcal E$ with edges $(i,j) \in \mathcal E$ weighted with their expected communication rate $\lambda_{ij}$, verifies $2\chi_1[\Lambda]\chi_2[\Lambda] \leq 1$. One such example is given by:
    \begin{equation*}\Lambda=\sqrt{2\chi_1[\mathcal{L}]\chi_2[\mathcal{L}]}\mathcal{L}\, .\end{equation*}
\end{remark}
This property allows us to use in DADAO the Laplacians introduced in previous work, to which we can now compare. It leads to the following proposition (proved in Appendix \ref{proof-optimality-dadao}), which shows that DADAO obtains better complexities than concurrent works while starting from the same family of Laplacians:

\begin{proposition}[Comparison with concurrent work]\label{optimality-dadao}\hfill\begin{itemize}
    \item If $\kappa=\mathcal{O}(\frac{|\bar{\mathcal{E}}|}{n})$, then DADAO obtains better communication rate than MSDA \cite{scaman17a}, and requires strictly less communications for the complete graph.
   \item For any fixed Laplacian valid for ADOM+ \cite{kovalev2021lower}, DADAO obtains a better communication rate than ADOM+, and requires strictly less communications for the complete graph.
   \item  If $\kappa =\mathcal{O}(\frac{|\bar{\mathcal{E}}|}{n})$, for a valid Laplacian using the Gossip matrix of OGT \cite{song2021optimal} or AGT \cite{li2021accelerated}, DADAO obtains a better communication rate than both Gradient Tracking methods, and requires strictly less communications for the  complete graph.
   \item DADAO requires fewer gradient computations than the Continuized framework \cite{even2021decentralized}, and has a strictly better computation rate for the cycle graph.
\end{itemize}
\end{proposition}

We highlight that \cite{scaman17a} claimed that their algorithm is optimal because they study the number of computations and \textit{synchronized} gossips on a worst-case graph; our claim is, \emph{by nature} different, as we are interested in the number of edges fired rather than the number of synchronized gossip rounds. Indeed, in an asynchronous framework, there is no notion of \textit{round of} communication, which allows this framework to enjoy the advantageous rates of randomized procedures. Moreover, not only this measure of complexity is standard in asynchronous frameworks (\eg, see \cite{Boyd2006}), but also, if we are aiming for the most frugal procedure, minimizing both the total number of computations and communications is of interest. Tab. \ref{complexity-graph} predicts the behavior of our algorithm on various classes of graphs encoded via a normalized Laplacian. It shows that systematically, our algorithm leads to the best complexities. For example, in the case of a complete graph, one synchronized gossip round requires a total of $\vert \mathcal E \vert = \mathcal O(n^2)$ communications whereas we show that a rate of only $\mathcal O (n)$ communications per time unit suffices for DADAO in this case. We note that the graph class depicted in Tab. \ref{complexity-graph} was used as worst-case examples for proving the optimality of \cite{scaman17a} in a synchronous context.

\begin{table*}[!ht]
    \centering
    \label{complexity-graph}
    \caption{Complexity for various graphs using $\frac 1 {\Vert \mathcal L \Vert}\mathcal L$ with $\mathcal L$ the standard Laplacian with unit edge weights. In this case, $\rho = \chi_1$.  The complexities are reported per time unit so that all algorithms reach $\epsilon$-precision at the same time. We have, respectively, for a star/line or cyclic/complete graph and the $d$-dimensional grid: $\chi_1=\mathcal{O}(n)$, $\chi_2=\mathcal{O}(n)$, $\mathrm{Tr} = \mathcal O (1)$ / $\chi_1=\mathcal{O}(n^2)$, $\chi_2=\mathcal{O}(1)$, $\mathrm{Tr} = \mathcal O (n)$  / $\chi_1=\mathcal{O}(1)$, $\chi_2=\mathcal{O}(1)$, $\mathrm{Tr} = \mathcal O (n)$ / $\chi_1=\mathcal{O}(n^{2/d})$, $\chi_2=\mathcal{O}(1)$, $\mathrm{Tr} = \mathcal O (n)$.}
\resizebox{2\columnwidth}{!}{
\begin{tabular}{l|cccc|cccc}
Method   & \multicolumn{4}{c}{\# edges activated per time unit} & \multicolumn{4}{c}{\# gradients computed per time unit}\\
\hline
Graph& Star & Line & Complete & $d$-grid& Star & Line & Complete & $d$-grid\\
\hline
\citep{kovalev2021lower} ADOM+&$n^{2}$&$n^3$&$n^2$&$n^{1+2/d}$&\cellcolor{green!25}$n$&\cellcolor{green!25}$n$&\cellcolor{green!25}$n$&\cellcolor{green!25}$n$\\
\citep{scaman17a} MSDA& $n^{3/2}$ & \cellcolor{green!25}$n^2$ & $n^2$ & \cellcolor{green!25}$n^{1+1/d}$&\cellcolor{green!25}$n$&\cellcolor{green!25}$n$&\cellcolor{green!25}$n$&\cellcolor{green!25}$n$\\
\citep{even2021continuized} Continuized&\cellcolor{green!25}$n$&\cellcolor{green!25}$n^2$&\cellcolor{green!25}$n$&\cellcolor{green!25}$n^{1+1/d}$&\cellcolor{green!25}$n$&$n^2$&\cellcolor{green!25}$n$&$n^{1+1/d}$\\
Centralized&\cellcolor{green!25}$n$&-&-&-&\cellcolor{green!25}$n$&-&-&-\\
DADAO (ours)&\cellcolor{green!25}$n$&\cellcolor{green!25}$n^2$&\cellcolor{green!25}$n$&\cellcolor{green!25}$n^{1+1/d}$&\cellcolor{green!25}$n$&\cellcolor{green!25}$n$&\cellcolor{green!25}$n$&\cellcolor{green!25}$n$
\end{tabular}}
\end{table*}

\section{Practical implementation}

\subsection{Algorithm}\label{algo}

To study the trajectories of $X_t\triangleq (x_t,\tilde x_t,\tilde y_t), Y_t\triangleq(y_t,z_t,\tilde z_t)$, we use the following, equivalent to \eqref{dynamic}:
\begin{align}dX_t&=a_1(X_t,Y_t)dt+b_1(X_t)d\mathbf{N}(t) \label{dynami-short} \\
        dY_t &=a_2(X_t,Y_t)dt+\sum_{(i,j)\in \mathcal{E}}b^{ij}_2(Y_t)dM_{ij}(t)\,, \nonumber
\end{align}
where $a_1,a_2,b_1=(b_1^i)_i,(b_2^{ij})_{ij}$ are smooth functions. 
We now describe the algorithm used to implement the dynamics of Eq. \eqref{dynamic}. Let us write $ T^{(i)}_1 < T^{(i)}_2 < ... < T^{(i)}_k <...$  the time of the $k$-th event on the $i$-th node, which is either an edge activation, either a gradient update. We remind that the spiking times of a specific event correspond to random variables with independent exponential increments and can thus be generated at the beginning of our simulation. They can also be generated on the fly and locally to stress the locality and asynchronicity of our algorithm. We write $X_t=(X^{(i)}_t)_i$ and $Y_t=(Y^{(i)}_t)_i$, then on the $i$-th node, at the $k$-th iteration, we integrate on $[T_{k}^{(i)};T_{k+1}^{(i)}]$ the  ODE \begin{align*}
dX_t&=a_1(X_t,Y_t)dt\\
dY_t&=a_2(X_t,Y_t)dt\,
\end{align*} to define the values right before the spike. One can easily find the matrix $\mathcal{A} \in \mathbb{R}^{6 \times 6}$ such that:
\begin{equation}\label{ODE}
\begin{pmatrix}
    X^{(i)}_{T^{(i)-}_{k+1}}\\
    Y^{(i)}_{T^{(i)-}_{k+1}}
\end{pmatrix}
=\exp\left((T^{(i)}_{k+1}-T^{(i)}_{k})\mathcal{A}\right)\begin{pmatrix}
    X^{(i)}_{T^{(i)}_{k}}\\
    Y^{(i)}_{T^{(i)}_{k}}
\end{pmatrix}\,.
\end{equation}
Next, if one has a gradient update, then:
$$X^{(i)}_{T^{(i)}_{k+1}}=X^{(i)}_{T^{(i)-}_{k+1}}+b_1\left(X^{(i)}_{T^{(i)-}_{k+1}}\right)\,.$$
Otherwise, if the edge $(i,j)$ or $(j,i)$ is activated, a communication bridge is created between both nodes $i$ and $j$. In this case,  the local update on $i$ writes:
\begin{align*}
Y^{(i)}_{T^{(i)}_{k+1}}&=Y^{(i)}_{T^{(i)-}_{k+1}}+b_2\left(Y^{(i)}_{T^{(i)-}_{k+1}}, Y^{(j)}_{T^{(i)-}_{k+1}}\right)\,.
\end{align*}

Note that, even if this event takes place along an edge $(i,j)$, we can write it separately for nodes $i$ and  $j$ by making sure they both have the events $T_{k_i}^{(i)}=T_{k_j}^{(j)}$, for some $k_i,k_j\in \mathbb{N}$ corresponding to this communication. As advocated, all those operations are local, and we summarize in the Alg. \ref{Algo} the algorithmic block corresponding to our implementation. See Appendix \ref{PracticalImplem} for more details.

\begin{algorithm}[ht!]
\caption{This algorithm block describes our implementation on each local machine. The $ODE$ routine is described by Eq. \ref{ODE}.}
\SetAlgoLined
\SetKwBlock{Init}{Initialize}{}
\SetKwProg{Para}{In parallel}{ continuously do:}{}
\SetKwComment{Comment}{// }{}
\label{Algo}

\KwIn{On each machine $i \in \{1,...,n\}$, gradient oracle $\nabla f_i$, parameters $\mu, L, \chi_1, t_{\max}$.}
\Init(on each machine $i \in \{1,...,n\}$:){
Set $X^{(i)}, Y^{(i)}, T^{(i)}$ to $0$ and set $\mathcal A$ via Eq. \eqref{def-A}\;}
\textbf{Synchronize} the clocks of all machines \;
\Para{on workers $i \in \{1,...,n\}$, \textup{\textbf{while}} $t < t_{\max}$, }{
$t \gets clock()$ \;
\If{there is an event at time $t$}{
$(X^{(i)},Y^{(i)}) \gets ODE(\mathcal{A},t-T^{(i)},X^{(i)},Y^{(i)})$\;
\If{the event is to take a gradient step}{
$X^{(i)} \gets X^{(i)}+b_1(X^{(i)})$ ;
}
\ElseIf{the event is to communicate with $j$}{
$Y^{(i)} \gets Y^{(i)}+b_2(Y^{(i)}, Y^{(j)})$
\tcp*[r]{\footnotesize{Happens at $j$ simultaneously.}}
}
$T^{(i)}\gets t$ \;
}
}
\Return{$(x_{t_{\max}}^{(i)})_{1\leq i\leq n}$.}
\end{algorithm}

\subsection{Numerical results}\label{num}
\begin{figure*}[ht]
     \begin{center}
        \subfigure{
            \includegraphics[width=1.\columnwidth]{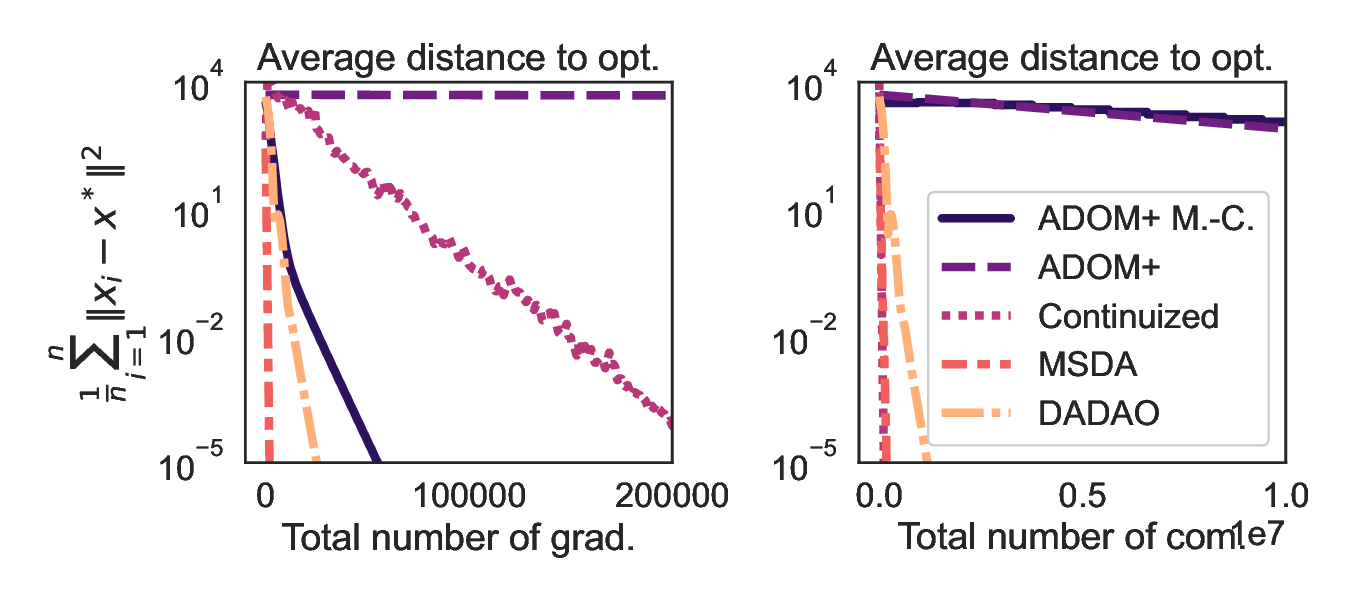}
        }
        \subfigure{
          \includegraphics[width=1.\columnwidth]{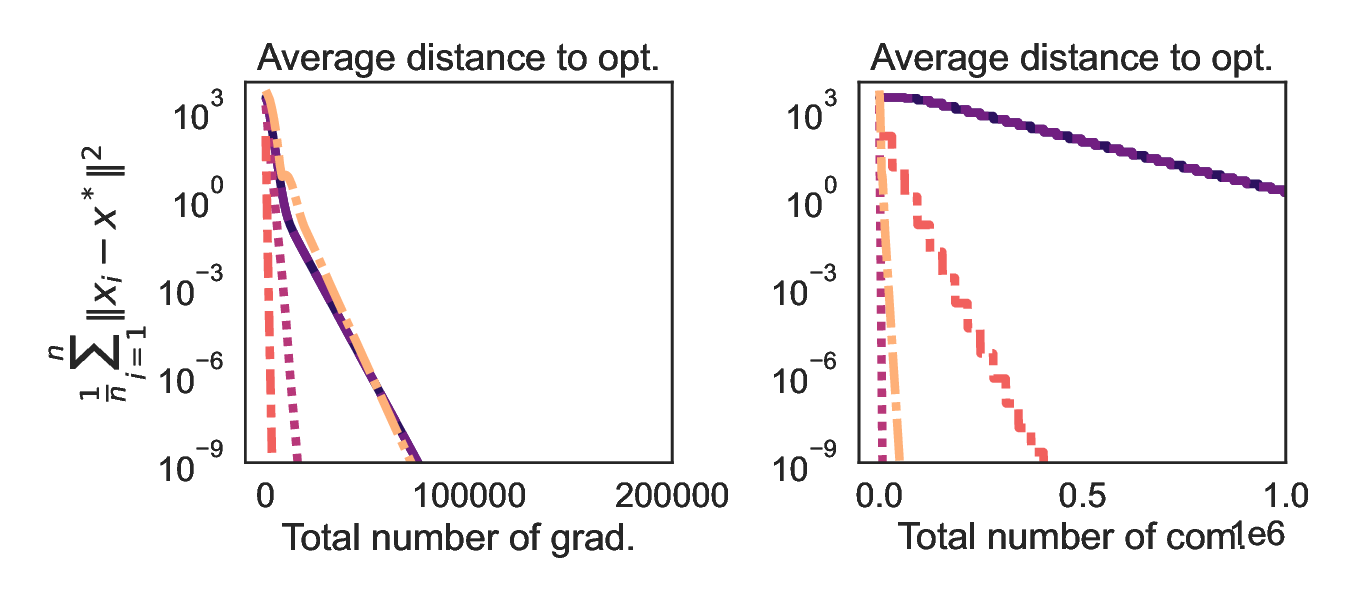}
        }
    \end{center}
    \caption{
    Comparison between ADOM+ \citep{kovalev2021lower}, the continuized framework \citep{even2021continuized}, MSDA \citep{scaman17a} and DADAO, using the same data for the linear regression task, and the same graphs \textit{(from left to right: line with $n=150$, complete with $n=250$)}.}
    \label{ExpeAll}
\end{figure*}
In this section, we study the behavior of our method in a standard experimental setting (\eg, see\cite{kovalev2021lower,even2021continuized}). In order to compare to methods using the gradient of the Fenchel conjugate \cite{even2021continuized} in our experiments, we restrict ourselves to a situation where it is easily computable. Thus, we perform the empirical risk minimization for the decentralized linear regression task given by:
\begin{equation}
\begin{aligned}
    f_i(x)&=\frac 1 m \sum_{j=1}^m \Vert a_{ij}^\top x - c_{ij} \Vert^2,
\end{aligned}
\end{equation}
where $a_{ij} \in \mathbb{R}^d$, and $c_{ij} \in \mathbb{R}$ correspond to $m$ local data points stored at node $i$. We follow a  protocol similar to \cite{kovalev2021lower}: we generate $n$ independent synthetic datasets with the \texttt{make\textunderscore regression} functions of scikit-learn \citep{scikit-learn}, each worker storing $m=100$ data points. We recall that the metrics of interest are the total number of local gradient steps and the total number of individual messages exchanged \textit{(\ie, number of edges that fired)} to reach an $\epsilon$-precision. We systematically used the proposed hyper-parameters of each reference paper for our implementation without any specific fine-tuning.

\paragraph{Comparison between all methods.}We fix the Laplacian matrix via Eq. \eqref{laplacian} to compare simultaneously to the continuized framework \citep{even2021continuized}, MSDA \citep{scaman17a} and ADOM+ \cite{kovalev2021lower}. We compare ourselves to both versions of ADOM+: with and without the Multi-Consensus (M.-C.). We report in Fig. \ref{ExpeAll} results corresponding to the complete graph with $n=250$ nodes and the line graph of size $n=150$. While sharing the same asymptotic rate (see Fig. \ref{star_msda_dadao} for experimental confirmation), we note that the Continuized framework~\citep{even2021continuized} and MSDA~\citep{scaman17a} have better absolute constants than DADAO, giving them an advantage both in terms of the number of communication and gradient steps. However, in the continuized framework, the gradient and communication steps being coupled, the number of gradient computations can potentially be orders of magnitude worse than our algorithm, which is reflected by Fig. \ref{ExpeAll} for the line graph. As for MSDA, Tab. \ref{complexity-graph} showed they do not have the best communication rates on certain classes of graphs, as  confirmed to the right in Fig. \ref{ExpeAll} for MSDA and in Fig. \ref{star_msda_dadao}. Thanks to its M.-C. procedure, ADOM+ can significantly reduce the number of necessary gradient steps. Yet, consistently with our analysis in Prop. \ref{optimality-dadao}, our method is systematically better in all settings in terms of communications.

\paragraph{Further comparison between DADAO and MSDA.} To experimentally confirm that our communication complexity is better than accelerated methods using the spectral gap and abstract away the better absolute constants for MSDA, we ran DADAO and MSDA on the task of distributed linear regression for star graphs of size $n \in \{10,20, 70, 200, 300, 1000, 2000\}$. We considered the evolution of the average distance to the optimal with the number of gradient steps and commmunication steps in log scale for each run, and computed the slope of each line. For each graph size, we report in Fig. \ref{star_msda_dadao} the rate between the slope for DADAO and the slope for MSDA.
\begin{figure}[ht]
     \begin{center}
    \includegraphics[width=1.0\columnwidth]{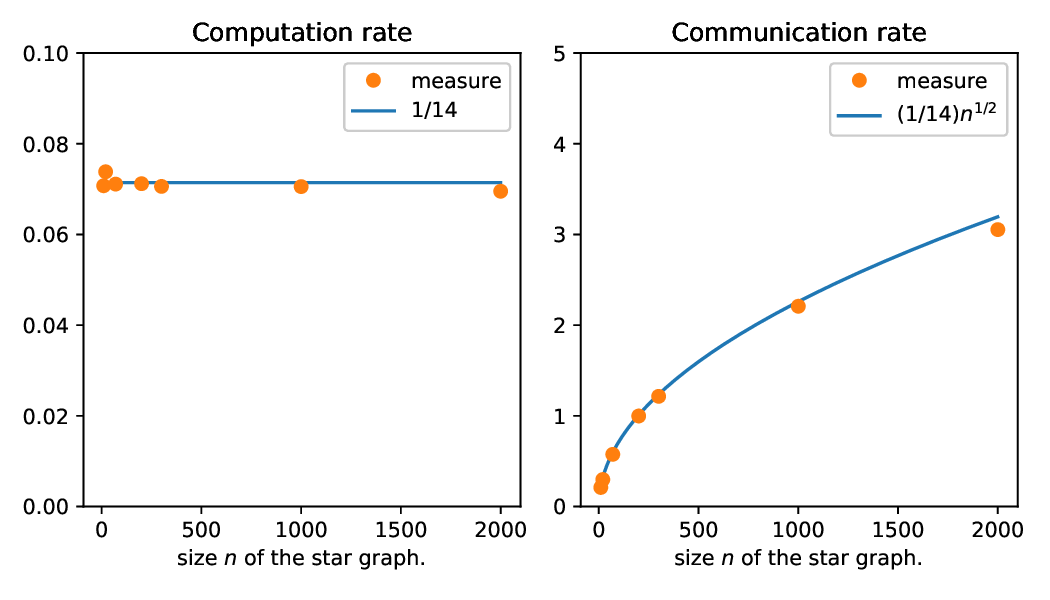}
    \end{center}
    \caption{
    Rate between the slopes in $\log$ scale of DADAO and MSDA for star graphs of size $n \in \{10,20, 70, 200, 300, 1000, 2000\}$.}
    \label{star_msda_dadao}
\end{figure}
We remark that the rate between the gradient complexities of DADAO and MSDA is indeed a $\mathcal O (1)$ (with a constant value of $\simeq 1/14$) while MSDA is indeed $\mathcal O (\sqrt{n} )$ worse than DADAO for communications on the star graph, as stated in Tab. \ref{complexity-graph}.

\paragraph{Discussion.} We do not claim the optimality of the absolute constant in DADAO's rate (the $8 \sqrt{2}$ term in the exponential of Th. \ref{main-thm}), which might be further optimized with a tighter analysis and different values of parameters. However, our focus being on large scale problems, we are concerned about asymptotic rates rather than the effect of the absolute constants. Thus, even-though  Fig. \ref{ExpeAll} may sometimes display a faster convergence for the continuized framework \citep{even2021continuized} and MSDA \citep{scaman17a} in terms of \textit{number of steps},  we remind that both \citep{even2021continuized, scaman17a} use expensive evaluation of dual gradients whereas we only use primal gradients, and Tab. \ref{complexity-graph} confirms that DADAO has at least their asymptotic rate. On the star graph, where Tab. \ref{complexity-graph} showed a strict advantage for DADAO compared to MSDA, Fig. \ref{star_msda_dadao} confirms that our convergence rate meets the one of \citep{scaman17a} for $n \simeq 200$, and has strictly better rates for larger graphs.
Finally, we note that, our algorithm falling into the \textit{randomized} category, further improvements may be theoretically possible. Indeed, \cite{Woodworth2016complexity} showed that the lower bound on the number of grad-and-prox oracle accesses for the optimization of composite objectives with deterministic methods is in $\mathcal{O}(n \sqrt{\frac L \mu} \log(\frac 1 \epsilon))$, which is attained by work such as \citep{scaman17a}. However, for randomized algorithms, the lower bound is in $ \mathcal{O}(n + \sqrt{\frac {nL} \mu} \log(\frac 1 \epsilon))$ and, to the best of our knowledge, no algorithm meeting this bound is known to date in the decentralized setting, and in particular, in the asynchronous one.

In conclusion, while several methods can share similar asymptotic convergence rates, ours is the only one to perform at least as well as its competitors in every setting for different graph's topology, as predicted by Tab. \ref{tab:other_methods}. 

\section{Conclusion}
In this work, we have proposed a novel stochastic algorithm for the decentralized optimization of a sum of smooth and strongly convex functions. We have demonstrated, theoretically and empirically, that this algorithm leads systematically to a substantial acceleration compared to state-of-the-art works. Furthermore, our algorithm is asynchronous, decoupled, primal, and does not rely on an extra inner loop: each of these properties makes it suitable for real applications. In future work, we would like to explore the robustness of such algorithms to more challenging variabilities occurring in real-life applications such as time-varying networks and to extend our work to less regular functions and stochastic analysis.

\section*{Acknowledgements}
This work was supported by Project ANR-21-CE23-0030 ADONIS and EMERG-ADONIS from Alliance SU. This work was granted access to the HPC/AI resources of IDRIS under the allocation AD011013743 made by GENCI. In addition, the authors would like to thank Raphaël Berthier, Mathieu Even, Hadrien Hendrikx, and Dmitry Kovalev for their helpful discussions.
\bibliography{ref}
\bibliographystyle{icml2023}

\newpage
\appendix
\onecolumn

\addcontentsline{toc}{section}{Appendix} 
\part{Appendix} 
\parttoc 


\section{Communication bounds}
The following properties will be used all along the proofs of the Lemmas and Theorems and are related to the communication of our nodes.
\begin{lemma}\label{stable-span}
Under the assumptions of Theorem \ref{main-thm}, if $z_0,\tilde z_0\in \text{span}(\pi)$, then $z_t,\tilde z_t\in \text{span}(\pi)$ almost surely.
\end{lemma}
\begin{proof}
It's clear that for any $i,j$, we get:
$$\pi (\mathbf{e}_i-\mathbf{e}_j)(\mathbf{e}_i-\mathbf{e}_j)^\mathsf{T}=(\mathbf{e}_i-\mathbf{e}_j)(\mathbf{e}_i-\mathbf{e}_j)^\mathsf{T}\,.$$
Thus, the variations of $(z_t,\tilde z_t)$ belong to $\text{span}(\pi)$, and so is the trajectory.
\end{proof}

\begin{lemma}[Effective resistance contraction]\label{resistance}For $(i,j) \in \mathcal E$ and any $x\in \mathbb{R}^{n\times d}$, we have:
$$\Vert (\mathbf{e}_i-\mathbf{e}_j)(\mathbf{e}_i-\mathbf{e}_j)^\mathsf{T}x\Vert_{\mathbf{\Lambda}^+}^2\leq  \chi_2 \Vert(\mathbf{e}_i-\mathbf{e}_j)(\mathbf{e}_i-\mathbf{e}_j)^\mathsf{T}x\Vert^2 \, .$$
\end{lemma}
\begin{proof}
Indeed, we note that:
\begin{align}
    \Vert (\mathbf{e}_i-\mathbf{e}_j)(\mathbf{e}_i-\mathbf{e}_j)^\mathsf{T}x\Vert_{\mathbf{\Lambda}^+}^2&= \langle (\mathbf{e}_i-\mathbf{e}_j)^\mathsf{T} x, (\mathbf{e}_i-\mathbf{e}_j)^\mathsf{T}\mathbf{\Lambda}^+(\mathbf{e}_i-\mathbf{e}_j)(\mathbf{e}_i-\mathbf{e}_j)^\mathsf{T}x \rangle \\
    &\leq 2  \chi_2 \langle (\mathbf{e}_i-\mathbf{e}_j)^\mathsf{T} x, (\mathbf{e}_i-\mathbf{e}_j)^\mathsf{T}x \rangle \\
    &=  \chi_2\Vert (\mathbf{e}_i-\mathbf{e}_j)(\mathbf{e}_i-\mathbf{e}_j)^\mathsf{T}x\Vert^2
\end{align}
\end{proof}

\begin{lemma}[Bound on the resistance]\label{resistance-up}\hfill

    For $(i,j)\in \mathcal{E}$, $(e_i-e_j)^\mathsf{T}\Lambda^+(e_i-e_j)\leq \frac{1}{\lambda_{ij}+\lambda_{ji}}$.
\end{lemma}

\begin{proof}
For a graph such that $\vert \bar{\mathcal E} \vert = 1$, the inequality is trivial. Now, we assume that there are other edges than $(i,j)$ or $(j,i)$. Thus, we have:

$$\Lambda \succcurlyeq (\lambda_{ij}+\lambda_{ji}) (e_i-e_j)(e_i-e_j)^\mathsf{T}+\tilde \lambda \tilde \pi$$
where $\tilde \pi (e_i-e_j)=0,\tilde \pi \mathbf{1}=0$, $\text{rank}(\tilde \pi) = n-1$, $\tilde \pi$ orthogonal projector and $\tilde \lambda>0$ small enough. In this case: $ \left((\lambda_{ij}+\lambda_{ji}) (e_i-e_j)(e_i-e_j)^\mathsf{T}+\tilde \lambda \tilde \pi \right)^+ \succcurlyeq \Lambda^+$,
but this implies that
$\frac 1 {\lambda_{ij}+\lambda_{ji}}\geq (e_i-e_j)^\mathsf{T}\Lambda^+(e_i-e_j)$.
\end{proof}

\begin{proof}[Proof of Lemma \ref{chi-lemma}] \label{proof-chi-lemma} First, we note that $\Lambda $ is symmetric and has a non-negative spectrum, as:
$$x^\mathsf{T} \Lambda  x = \sum_{(ij) \in \mathcal{E}} \lambda_{ij} \Vert x_i - x_j \Vert^2 \, .$$
From this, we also clearly see that $\chi_1 = + \infty $ iff the graph is disconnected. Next, assuming that the graph is connected, $0$ is an eigenvalue of $\Lambda$ with multiplicity $1$ and by definition of $\chi_1 $, we have $\chi_1\geq \chi_2$. As we also have the following:
$$
\sum_{(i,j) \in \mathcal E } \lambda_{ij}(e_i-e_j)^\mathsf{T}\Lambda^+(e_i-e_j) = \mathrm{Tr}\,( \Lambda^+ \Lambda ) = n-1 \, ,
$$
we can write:
$$
n-1 \leq 2\chi_2 \sum_{(i,j) \in \mathcal E } \lambda_{ij} = \chi_2\mathrm{Tr}\,\Lambda
$$
and get $\frac {n-1} {\mathrm{Tr}\,\Lambda } \leq\chi_2$.
Finally, note that: $\mathrm{Tr}\,(\Lambda) = 2\sum_{(i,j) \in \mathcal{E}} \lambda_{ij} \leq 2 \vert \bar{\mathcal{E}}\vert \sup_{(i,j)\in\mathcal{E}}( \lambda_{ij} + \lambda_{ji} )$ and $\mathrm{Tr}\,(\Lambda) \leq (n-1)\Vert \Lambda\Vert$, so that, using Lemma \ref{resistance-up}:

\begin{align}
\sqrt{\chi_2}\mathrm{Tr }(\Lambda)&\leq \frac{1}{\sqrt{2\inf_{(i,j)\in\mathcal{E}}(\lambda_{ij} + \lambda_{ji} )}} \sqrt{\mathrm{Tr}\, \Lambda} \sqrt{\mathrm{Tr}\, \Lambda}\\
&\leq \sqrt{\frac{\mathrm{Tr}\, \Lambda}{2\inf_{(i,j)\in\mathcal{E}}(\lambda_{ij} + \lambda_{ji} )}} \sqrt{2 \vert \bar{\mathcal{E}}\vert \sup_{(i,j)\in\mathcal{E}} (\lambda_{ij}  + \lambda_{ji} )}\\
& \leq  \sqrt{\kappa} \sqrt{\vert \bar{\mathcal{E}}\vert } \sqrt{(n-1)\Vert \Lambda\Vert} 
\end{align}

\end{proof}

\section{Saddle Point Reformulation}
\label{saddle-point-reform}
With $0 <\nu<\mu$ and introducing an extra dual variable $\hat x$, we get:
\begin{align*}
\inf_{x\in \mathbb{R}^d} \sum_{i=1}^n f_i(x) &=\inf_{\substack{ x, \hat x \in \mathbb{R}^{n \times d} \\x=\hat x, \pi \hat x=0}} \sum_{i=1}^n f_i(x_i)-\frac \nu 2 \Vert x\Vert^2 +\frac \nu 2 \Vert \hat x\Vert^2\\
&=\inf_{x,\hat x \in \mathbb{R}^{n \times d}}\sup_{y,z  \in \mathbb{R}^{n \times d}}\sum_{i=1}^n f_i(x_i)-\frac \nu 2 \Vert x\Vert^2 +\frac \nu 2 \Vert \hat  x\Vert^2+ \langle y,\hat x-x \rangle +\langle z,\pi \hat x \rangle\\
&=\inf_{x \in \mathbb{R}^{n \times d}}\sup_{y,z  \in \mathbb{R}^{n \times d}} \inf_{\hat x \in \mathbb{R}^{n \times d}}\sum_{i=1}^n f_i(x_i)-\frac \nu 2 \Vert x\Vert^2 +\frac \nu 2 \Vert \hat  x\Vert^2+ \langle y,\hat x-x \rangle +\langle z,\pi \hat x \rangle\\
&=\inf_{x \in \mathbb{R}^{n \times d}}\sup_{y,z  \in \mathbb{R}^{n \times d}} \sum_{i=1}^nf_i(x_i)-\frac \nu 2\Vert x\Vert^2- \langle x,y \rangle-\frac 1{2\nu}\Vert \pi z+ y\Vert^2\,.
\end{align*}
\section{Proof of the main Theorem (Th. 3.5) }
\label{proof-main-thm}

\paragraph{Basic reminds about the Bregman divergence}
 For a smooth convex function $F$, $$d_F(x,y)\triangleq F(x)-F(y)-\langle \nabla F(y),x-y\rangle$$ is its Bregman divergence. Next, we  set $\nu=\frac \mu 2$ such that, for $F(x)=\sum_{i=1}^nf_i(x_i)-\frac \nu2 \Vert x\Vert^2$, then $F$ is $L$-smooth, $\nu$-strongly convex, and we get:
$$\frac{1}{2L}\Vert \nabla F(x)-\nabla F(y)\Vert^2\leq d_F(x,y)\leq\frac L2 \Vert x-y\Vert^2\,,$$
and
$$
\frac{\nu}{2}\Vert x-y\Vert^2\leq d_F(x,y)\leq\frac 1{2\nu} \Vert \nabla F(x)-\nabla F(y)\Vert^2\,,
$$

\begin{proof}[Proof of Theorem \ref{main-thm}]
We introduce for a positive semi-definite matrix $A$, $\Vert x\Vert_A^2\triangleq \langle x,Ax \rangle$. For our proof, we rely on the notation of Eq. \refeq{dynami-short}, and we introduce $X\triangleq(x,\tilde x,\tilde y),Y\triangleq(y,z,\tilde z)$ and the following Lyapunov potential:
\begin{align*}
\Phi(t,X,Y)&\triangleq A_t d_F(x,x^*) + \tilde A_t\Vert \tilde x-x^*\Vert^2 +B_t \Vert y-y^*\Vert^2+\tilde B_t\Vert \tilde y-y^*\Vert^2\\
&+C_t\Vert z+y-z^*-y^*\Vert^2+\tilde C_t\Vert \tilde z-z^*\Vert_{\mathbf{\Lambda}^+}^2\,,
\end{align*}
where $A_t,\tilde A_t,B_t,\tilde B_t,C_t,\tilde C_t$ are non-negative functions to be defined. We will use this potential to control the trajectories.

Because $\Phi$ is smooth, the SDE is a smooth trajectory,  we get via Ito's formula~\citep{last2017lectures} applied to the semi-martingale $(X_t,Y_t)$ on any intervals $[0,T]$:
\begin{align*}
\Phi(T,X_{T},Y_{T})&=\Phi(0,X_{0},Y_{0})+\int_{0}^{T}\langle \nabla \Phi(t,X_t,Y_t),\begin{pmatrix}1\\a_1(X_t,Y_t)\\a_2(X_t,Y_t)\end{pmatrix}\rangle dt\\
&+\sum_{i=1}^n\int_{0}^{T}\big(\Phi(t,X_t+b^i_1(X_t),Y_t)-\Phi(t,X_t,Y_t)\big)dt\\
&+\sum_{(i,j)\in\mathcal{E} }\int_{0}^{T}\big(\Phi(t,X_t,Y_t+b^{ij}_2(Y_t))-\Phi(t,X_t,Y_t)\big)\lambda_{ij}dt\\ &+\Theta_{T}\,,
\end{align*}
where the following quantity is a Martingale:
\begin{align*}
\Theta_u&\triangleq\sum_{i=1}^n\int_0^u\big(\Phi(t,X_{t^-},Y_{t^-}+b^i_1(X_{t^-}))-\Phi(u,X_{t^-},Y_{t^-})\big)(dN_i(t)-dt)\\
&+\sum_{(i,j)\in\mathcal{E} }\int_0^u\big(\Phi(t,X_{t^-}+b^{ij}_2(X_{t^-}),Y_{t^-})-\Phi(t,X_{t^-},Y_{t^-})\big)(dM_{ij}(t)-\lambda_{ij}dt)\,.\end{align*}
Taking the expectation, we get that, as the initialization is deterministic:

\begin{align*}
\mathbb{E}[\Phi(T,X_{T},Y_{T})]&=\Phi(0,X_{0},Y_{0})+\int_{0}^{T}\langle \nabla \Phi(t,X_t,Y_t),\begin{pmatrix}1\\a_1(X_t,Y_t)\\a_2(X_t,Y_t)\end{pmatrix}\rangle dt\\
&+\sum_{i=1}^n\int_{0}^{T}\big(\Phi(t,X_t+b^i_1(X_t),Y_t)-\Phi(t,X_t,Y_t)\big)dt\\
&+\sum_{(i,j)\in\mathcal{E} }\int_{0}^{T}\big(\Phi(t,X_t,Y_t+b^{ij}_2(Y_t))-\Phi(t,X_t,Y_t)\big)\lambda_{ij}dt\,,
\end{align*}

 To show that the integrand term is negative, we will use the following technical Lemma, which is also difficult to prove and whose proof is deferred to Appendix \ref{proof-important-lemma}:
\begin{lemma}\label{important-lemma}
If:
$$\begin{array}{llllll}
        \eta  = \frac 18 \sqrt{\frac \nu L}
         &  \; \gamma = \frac{1}{4L}
         &  \;\delta = \frac 14 \sqrt{\frac \nu L}
         & \;\alpha = \frac 14 \sqrt{\frac \nu L}
         & \;\beta = \frac 12
         &  \;\theta = \frac 12 \sqrt{\frac L \nu }
         \\
         \tilde \eta = \frac 18 \sqrt{\frac \nu L}
         & \; \tilde \gamma = \frac 1{4 \sqrt{\nu L}}
         & \; \tilde \delta = 1
         & \;\tilde \alpha = \frac 18 \sqrt{\frac \nu L}
         & \;\tilde \beta = 2 \chi_1[\Lambda] \sqrt{\frac L \nu }
         & \; \nu = \frac \mu 2
    \end{array}$$

and

$$\begin{array}{lllllll}
A_t=e^{\frac{t}{8\sqrt 2}\sqrt{\frac{\mu}L}} & \; \tilde A_t=\frac{\mu}8e^{\frac{t}{8\sqrt 2}\sqrt{\frac{\mu}L}} & \; \tilde B_t=\frac{1}{8L}e^{\frac{t}{8\sqrt 2}\sqrt{\frac{\mu}L}} \\
B_t=\frac{1}{16L}e^{\frac{t}{8\sqrt 2}\sqrt{\frac{\mu}L}}  & \; C_t=\frac{1}{2\mu}e^{\frac{t}{8\sqrt 2}\sqrt{\frac{\mu}L}}  & \; \tilde C_t=\frac{1}{32\chi_1 L}e^{\frac{t}{8\sqrt 2}\sqrt{\frac{\mu}L}}\,.
\end{array}$$

then:

\begin{align*}
\langle \nabla &\Phi(t,X_t,Y_t),\begin{pmatrix}1\\a_1(X_t,Y_t)\\a_2(X_t,Y_t)\end{pmatrix}\rangle 
+\big(\Phi(t,X_t+b_1(X_t),Y_t)-\Phi(t,X_t,Y_t)\big)\\
&+\sum_{(i,j)\in \mathcal{E} }\lambda_{ij} \big(\Phi(t,X_t,Y_t+b^{ij}_2(Y_t))-\Phi(t,X_t,Y_t)\big)\leq 0\text{ a.s. .}
\end{align*}

\end{lemma}

Now, we remark that if we have $F(x) = f(x)-\frac \mu 4 \Vert x \Vert^2$, and initialize with $\tilde x_0=x_0$, $y_0=\tilde y_0=\nabla F(x_0)$ and $z_0=\tilde z_0=-\pi \nabla F(x_0)$, then, given the linear relation between $A_t,\tilde A_t,B_t,\tilde B_t,C_t,\tilde C_t$, the $L$ smoothness and the fact $\pi$ is an orthogonal projection, we get:
\begin{align*}
    \Phi(0,X_0,Y_0)&\leq d_F(x_0,x^*)+\frac{\mu}8 \Vert x_0-x^*\Vert^2+\frac{1}{16L} \Vert \nabla F(x_0)-\nabla F(x^*)\Vert^2\\
    &+\frac{1}{8L} \Vert \nabla F(x_0)-\nabla F(x^*)\Vert^2
    +\frac 1{2\mu}\Vert (\mathbf{I}-\pi)(\nabla F(x_0)-\nabla F(x^*))\Vert^2 \\
    &+\frac{1}{32}\frac{\chi_1}{L \chi_1}\Vert \pi (\nabla F(x_0)-\nabla F(x^*))\Vert^2\\
    &\leq (\frac{L}2+\frac \mu 8+\frac L{16}+\frac L8+\frac {L^2}{2\mu}+\frac L{32})\Vert x_0-x^*\Vert^2
\end{align*}

In particular, as $\frac{\mu}4\Vert x-x^*\Vert^2\leq d_F(x,x^*)$ it implies that:

$$\mathbb{E}[\Vert x_t-x^*\Vert^2]\leq (\frac 12+\frac {23}8\frac{L}\mu+2\frac {L^2}{\mu^2})\Vert x_0-x^*\Vert^2e^{-\frac{t}{8\sqrt 2}\sqrt{\frac{\mu}L}} $$

Finally, we note that the expected number of gradients between $[0,T]$ is given by:

$$\mathbb{E}[\sum_{i=1}^n N_i(T)]=nT\,,$$

and similarly, the number of edges activated is given by:
$$\mathbb{E}[\sum_{1\leq i,j\leq n}M_{ij}(T)]=\sum_{1\leq i,j\leq n}\int_0^T\lambda_{ij} \,dt=\frac T2\mathrm{Tr}\,\Lambda.$$
\end{proof}

\subsection{Proof of the Lemma C.1}
\label{proof-important-lemma}
We first state a couple of inequalities that we will combine to obtain a bound on our Lyapunov function. In all this section, for a given variable $x$, we denote by $x^+$ the value of $x$ right after a Poisson update.
\begin{lemma}
First:
\begin{align}
\phi_A&\triangleq A_t(d_F(x^+,x^*)-d_F(x,x^*))+\tilde A_t(\Vert \tilde x^+-x^*\Vert^2-\Vert \tilde x-x^*\Vert^2) \notag\\
&+\eta A_t\langle  \tilde x-x,\nabla F(x)-\nabla F(x^*)\rangle +2\tilde\eta\tilde A_t\langle x- \tilde x,\tilde x-x^*\rangle\\
& \leq \Vert \nabla F(x)- \tilde y\Vert^2 \left( 
A_t \frac{L\gamma^2}{2} - A_t  \gamma  + \tilde A_t \tilde \gamma^2
\right) \notag\\
& +A_t\gamma \langle \nabla F(x)-\tilde y,y^*-\tilde y\rangle + 2\tilde\gamma \tilde A_t\langle \tilde y-y^*,\tilde x-x^*\rangle \\
& -2\tilde\gamma \tilde A_t \left(d_F(\tilde x,x^*)+d_F(x^*,x)-d_F(\tilde x,x) \right)\notag\\
&-\eta A_t(d_F(\tilde x,x)+d_F(x,x^*)-d_F(\tilde x,x^*)) -\tilde A_t\tilde \eta\Vert \tilde x-x^*\Vert^2+ \tilde A_t\tilde\eta \Vert x-x^*\Vert^2 \notag
\end{align}
\end{lemma}
\begin{proof}
First, we have to use optimality conditions and smoothness, as well as the separability of $F$:
\begin{align}
d_F(x^+,x^*)-d_F(x,x^*)&= d_F(x^+,x)-\langle x^+-x,\nabla F(x^*)-\nabla F(x)\rangle\\
&\leq \frac{L}2\Vert x^+-x\Vert^2-\langle x^+-x,\nabla F(x^*)-\nabla F(x)\rangle\\
&=\frac{L\gamma^2}{2}\Vert \tilde y-\nabla F(x)\Vert^2- \gamma \Vert \nabla F(x)- \tilde y\Vert^2 \notag\\
&+ \gamma \langle \nabla F(x)-\tilde y,y^*-\tilde y\rangle
\end{align}
Next, we note that, again using optimality conditions:
\begin{align}\Vert \tilde x^+-x^*\Vert^2-\Vert \tilde x- x^*\Vert^2&=2\langle \tilde x^+-\tilde x,\tilde x-x^*\rangle+\Vert \tilde x^+-\tilde x\Vert^2 \\
&=-2\tilde\gamma\langle \nabla F(x)- \tilde y,\tilde x-x^*\rangle + \tilde \gamma^2\Vert \nabla F(x)- \tilde y\Vert^2\\
&=-2\tilde\gamma \langle \nabla F(x)-\nabla F(x^*),\tilde x-x^*\rangle  + 2\tilde\gamma\langle \tilde y-y^*,\tilde x-x^*\rangle+\tilde \gamma^2\Vert \nabla F(x)- \tilde y\Vert^2\\
&=-2\tilde\gamma (d_F(\tilde x,x^*)+d_F(x^*,x)-d_F(\tilde x,x)) + 2\tilde\gamma \langle \tilde y-y^*,\tilde x-x^*\rangle +\tilde \gamma^2\Vert \nabla F(x)-\tilde y\Vert^2
\end{align}
Momentum in $x$ associated with the term $d_F(x, x^*)$ gives:
\begin{equation}
    \eta \langle  \tilde x-x,\nabla F(x)-\nabla F(x^*)\rangle =-\eta (d_F(\tilde x,x)+d_F(x,x^*)-d_F(\tilde x,x^*))
\end{equation}
and momentum in $\tilde x$ associated with $\Vert \tilde x - x^* \Vert^2$ leads to:
\begin{equation}
    2\tilde\eta\langle x- \tilde x,\tilde x-x^*\rangle =- 2\tilde\eta\Vert \tilde x-x^*\Vert^2+ 2\tilde\eta \langle x-x^*,\tilde x-x^*\rangle\leq -\tilde \eta\Vert \tilde x-x^*\Vert^2+\tilde\eta \Vert x-x^*\Vert^2
\end{equation}
\end{proof}

\begin{lemma}
Next, we show that if $\alpha B_t =\frac \delta 2 \tilde B_t$:
\begin{align}
    \phi_B&\triangleq  B_t(\Vert  y^+-y^*\Vert^2 - \Vert y - y^* \Vert^2) +\tilde B_t(\Vert \tilde y^+-y^*\Vert^2 - \Vert \tilde y - y^* \Vert^2) \notag\\
    & \quad +2 \alpha B_t \langle y-y^*,\tilde y -y\rangle-2 \theta \tilde B_t \langle y+z+\nu \tilde x, \tilde y - y^* \rangle + 2 \alpha C_t \langle \tilde y - y, z+y - y^* -z^* \rangle \\
    &\leq -\frac{\delta}{2} \tilde B_t\Vert \tilde y-y^*\Vert^2-\frac{\delta}{2} \tilde B_t \Vert  y-y^*\Vert^2-2\tilde \delta \tilde B_t\langle \nabla F(x)-\tilde y,y^*-\tilde y\rangle \notag\\
   & \quad + \delta  \tilde B_t \Vert \nabla F(x) -\nabla F(x^*) \Vert^2 +  \left((\delta + \tilde \delta)^2-\delta \right)\tilde B_t \Vert \nabla F(x)-y\Vert^2 \notag\\
   & \quad -2 \theta \tilde B_t \langle y + z  - y^* - z^*, \tilde y-y^* \rangle - 2 \theta \nu \tilde B_t\langle \tilde x - x^*, \tilde y-y^* \rangle + 2 \alpha C_t \langle \tilde y - y, z+y - y^* -z^* \rangle
\end{align}
\end{lemma}
\begin{proof}
 Using optimality conditions:
\begin{align}
    \Vert \tilde y^+-y^*\Vert^2-\Vert \tilde y-y^*\Vert^2
    &=2\langle \tilde y-y^*,\tilde y^+-\tilde y\rangle+\Vert \tilde y^+-\tilde y\Vert^2\\
    &= 2\delta\langle \nabla F(x) - \tilde y,\tilde y-y^*\rangle+2\tilde \delta\langle \nabla F(x) - \tilde y, \tilde y-y^*\rangle + (\delta + \tilde \delta)^2\Vert \nabla F(x)- \tilde y\Vert^2\\
    &= -2 \tilde \delta\langle \nabla F(x) - \tilde y,y^* - \tilde y \rangle + \delta  \Vert \nabla F(x) -\nabla F(x^*) \Vert^2 - \delta \Vert \tilde y - y^* \Vert^2 \notag\\
   &  \quad + \left((\delta + \tilde \delta)^2-\delta \right)\Vert \nabla F(x)- \tilde y\Vert^2
\end{align}
The momentum in $\tilde y$ associated with the term $\Vert \tilde y - y^*\Vert^2$ gives:
\begin{align}
    -2 \theta \tilde B_t \langle y+z+\nu \tilde x, \tilde y - y^* \rangle =& -2 \theta \tilde B_t \langle y + z  - y^* - z^*, \tilde y-y^* \rangle - 2 \theta \nu \tilde B_t\langle \tilde x - x^*, \tilde y-y^* \rangle
\end{align}
The momentum in $y$ associated with the term $\Vert y - y^*\Vert^2$ gives:
\begin{align}
    2 \alpha B_t \langle \tilde y - y, y - y^*\rangle &= -  \alpha B_t \Vert  y - y^* \Vert^2   -  \alpha B_t \Vert  \tilde y - y \Vert^2 +  \alpha B_t \Vert  \tilde y - y^* \Vert^2
\end{align}
and the one associated with $\Vert  y+z - y^* - z^*\Vert^2$:
\begin{align}
    2 \alpha C_t \langle \tilde y - y, z+y - y^* -z^* \rangle
\end{align}
\end{proof}

\begin{lemma}
Finally, assuming $\theta \tilde B_t =  \tilde \beta \tilde C_t =  \alpha C_t$, letting $1 \geq \tilde \tau >0$, \\ $z_{ij}^+=z-\beta (\mathbf{e}_i-\mathbf{e}_j)(\mathbf{e}_i-\mathbf{e}_j)^\mathsf{T}(y+z)$ and $\tilde z_{ij}^+=\tilde z-\tilde \beta(\mathbf{e}_i-\mathbf{e}_j)(\mathbf{e}_i-\mathbf{e}_j)^\mathsf{T}(y+z)$:
\begin{align}
    \phi_{C} &
    -2 \theta \tilde B_t \langle y + z  - y^* - z^*, \tilde y-y^* \rangle + 2 \alpha C_t \langle \tilde y - y, z+y - y^* -z^* \rangle
    \triangleq \notag \\
    & \quad \sum_{ij}\lambda_{ij}  C_t \Big(\Vert y + z_{ij}^+  - y^* - z^* \Vert^2-\Vert y+z - y^* -z^*\Vert^2 \Big) \notag \\
    & + \sum_{ij}\lambda_{ij}  \tilde C_t \Big(\Vert \tilde z_{ij}^+-z^*\Vert_{\mathbf{\Lambda}^+}^2-\Vert \tilde z-z^*\Vert_{\mathbf{\Lambda}^+}^2\Big) + 2 \tilde \alpha \tilde C_t \langle z - \tilde z, \tilde z- z^*\rangle_{\mathbf{\Lambda}^+} \\
    &+2 \alpha C_t \langle \tilde z +\tilde y- z - y, z+y - y^* -z^* \rangle-2 \theta \tilde B_t \langle y + z  - y^* - z^*, \tilde y-y^* \rangle \notag\\
    &\leq  \tilde \beta^2 \chi_2 \tilde C_t\sum_{(i,j)\in \mathcal{E} }\lambda_{ij} \Vert (\mathbf{e}_i-\mathbf{e}_j)(\mathbf{e}_i-\mathbf{e}_j)^\mathsf{T}(y+z)\Vert^2 \notag\\
    &+ \beta(\beta-1) C_t \sum_{(i,j)\in \mathcal{E} }\lambda_{ij} \Vert (\mathbf{e}_i-\mathbf{e}_j)(\mathbf{e}_i-\mathbf{e}_j)^\mathsf{T}(y+z) \Vert^2 \notag\\
    &- \alpha C_t \Vert y + z  - y^* - z^* \Vert^2+ \tilde \alpha\chi_1 \tilde C_t \Vert z-z^*\Vert^2-\tilde\alpha \tilde C_t \Vert\tilde z-z^*\Vert_{\mathbf{\Lambda}^+}^2 \notag\\
    &-\tilde \tau\frac 12 \tilde \beta \frac{\nu}L \tilde C_t \Vert z-z^*\Vert^2+  \tilde \tau \frac{\nu}L \frac{2\alpha \theta}{\delta} B_t\Vert y-y^*\Vert^2
\end{align}
\end{lemma}
\begin{proof}

Having in mind that $\pi(y^*+z^*)=0$ and $\mathbf{\Lambda}^+\mathbf{\Lambda}=\pi$, we get, using Lemma \ref{stable-span} and Lemma \ref{resistance} on the inequality \eqref{ineq1}:
\begin{align}
    \Delta_{\tilde z}&\triangleq \sum_{(i,j)\in \mathcal{E} }\lambda_{ij} \big(\Vert \tilde z_{ij}^+-z^*\Vert_{\mathbf{\Lambda}^+}^2-\Vert \tilde z-z^*\Vert_{\mathbf{\Lambda}^+}^2\big)\\
    &=\sum_{(i,j)\in \mathcal{E} }\lambda_{ij} 2\langle \tilde z-z^*, \tilde z_{ij}^+- \tilde z\rangle_{\mathbf{\Lambda}^+}+\Vert \tilde z_{ij}^+-\tilde z\Vert_{\mathbf{\Lambda}^+}^2\\
    &=-2\tilde\beta\sum_{(i,j)\in \mathcal{E} }\lambda_{ij} \langle \tilde z-z^*,(\mathbf{e}_i-\mathbf{e}_j)(\mathbf{e}_i-\mathbf{e}_j)^\mathsf{T}(y+z-y^*-z^*)\rangle_{\mathbf{\Lambda}^+} \notag \\
    & \quad + \sum_{(i,j)\in \mathcal{E} }\lambda_{ij} \tilde \beta^2\Vert (\mathbf{e}_i-\mathbf{e}_j)(\mathbf{e}_i-\mathbf{e}_j)^\mathsf{T}(y+z)\Vert_{\mathbf{\Lambda}^+}^2\\
    &\leq -2\tilde \beta\langle \tilde z-z^*,\mathbf{\Lambda}^+\mathbf{\Lambda} (y+z)\rangle  + \chi_2\tilde \beta^2\sum_{(i,j)\in \mathcal{E} }\lambda_{ij} \Vert (\mathbf{e}_i-\mathbf{e}_j)(\mathbf{e}_i-\mathbf{e}_j)^\mathsf{T}(y+z)\Vert^2\label{ineq1}\\
    &=-2\tilde\beta\langle\tilde z-z^*,\pi(y+z)\rangle +  \chi_2\tilde \beta^2\sum_{(i,j)\in \mathcal{E} }\lambda_{ij} \Vert (\mathbf{e}_i-\mathbf{e}_j)(\mathbf{e}_i-\mathbf{e}_j)^\mathsf{T}(y+z)\Vert^2
\end{align}
Noting that $y^+$, the value of $y$ after a Poisson update, is equal to $y$:
\begin{align}
    \Delta_z&\triangleq\sum_{(i,j)\in \mathcal{E} }\lambda_{ij} (\Vert y^+ + z_{ij}^+ - y^* - z^* \Vert^2 - \Vert y+z - y^* -z^*\Vert^2)\\
    &= 2\sum_{(i,j)\in \mathcal{E} }\lambda_{ij}  \langle y + z_{ij}^+ - y -z, y + z - y^* - z^* \rangle + \sum_{(i,j)\in \mathcal{E} }\lambda_{ij} \Vert y + z_{ij}^+ - y - z \Vert^2 \\
    & = -2 \sum_{(i,j)\in \mathcal{E} }\beta\lambda_{ij} \langle (\mathbf{e}_i-\mathbf{e}_j)(\mathbf{e}_i-\mathbf{e}_j)^\mathsf{T}(y+z), y + z - y^* - z^* \rangle \notag\\
    & \quad + \sum_{(i,j)\in \mathcal{E} }\beta^2\lambda_{ij} \Vert (\mathbf{e}_i-\mathbf{e}_j)(\mathbf{e}_i-\mathbf{e}_j)^\mathsf{T}(y+z) \Vert^2 \\  
&= \beta(\beta-1)\sum_{(i,j)\in \mathcal{E} }\lambda_{ij} \Vert (\mathbf{e}_i-\mathbf{e}_j)(\mathbf{e}_i-\mathbf{e}_j)^\mathsf{T}(y+z) \Vert^2
\end{align}
The momentum in $\tilde z$ associated with $\Vert \tilde z - z^*\Vert^2_{\mathbf{\Lambda}^+}$ gives:
\begin{align}
    2 \tilde \alpha \tilde C_t \langle z - \tilde z, \tilde z- z^*\rangle_{\mathbf{\Lambda}^+}
    &\leq \tilde \alpha  \chi_1 \tilde C_t \Vert z-z^*\Vert^2-\tilde\alpha \tilde C_t \Vert\tilde z-z^*\Vert_{\mathbf{\Lambda}^+}^2
\end{align}
And the one in $z$ associated with $\Vert  y+z - y^* - z^*\Vert^2$ gives:
\begin{align}
    2 \alpha C_t \langle \tilde z - z, z+y - y^* -z^* \rangle
\end{align}
Then, assuming that $\theta \tilde B_t = \tilde\beta \tilde C_t = \alpha C_t$, we have:
\begin{align}
    & 2 \alpha C_t \langle \tilde y - y, z+y - y^* -z^* \rangle
    -2  \tilde \beta \tilde C_t \langle \tilde z-z^*,y+z-y^*-z^*\rangle \notag\\
    &-2 \theta \tilde B_t \langle y + z  - y^* - z^*, \tilde y-y^* \rangle + 2 \alpha C_t \langle \tilde z - z, z+y - y^* -z^* \rangle\\
    & = -2 \alpha C_t \Vert y + z  - y^* - z^* \Vert^2 \label{yplusz}
\end{align}
At this stage, we split the negative term \eqref{yplusz} into two halves, upper-bounding one of the halves by remembering that $\frac \nu L \leq 1$ and introducing $1\geq\tilde \tau>0$:
\begin{align}
-\alpha C_t\Vert y+z-y^*-z^*\Vert^2\leq &-\tilde \tau\frac{\nu}L\alpha C_t\Vert y+z-y^*-z^*\Vert^2\\
&=-\tilde \tau \tilde \beta  \frac{\nu}L \tilde C_t\Vert y+z-y^*-z^*\Vert^2\\
&\leq -\tilde \tau\frac 12 \tilde \beta \frac{\nu}L \tilde C_t \Vert z-z^*\Vert^2+\tilde \tau  \tilde \beta  \frac{\nu}L \tilde C_t \Vert y-y^*\Vert^2\\
&=-\tilde \tau\frac 12 \tilde \beta \frac{\nu}L \tilde C_t \Vert z-z^*\Vert^2+  \tilde \tau \frac{\nu}L \frac{2\alpha \theta}{\delta} B_t\Vert y-y^*\Vert^2
\end{align}
\end{proof}

Keeping in mind that $\theta \tilde B_t = \tilde \beta \tilde C_t = \alpha C_t$ and $\frac \delta 2 \tilde B_t=\alpha B_t$, we put everything together. Defining $\Psi = \frac {\partial \Phi}{\partial t}+\phi_A + \phi_B +  \phi_C$, we have:

\begin{align}
 \Psi & \leq   \Vert \nabla F(x)- \tilde y\Vert^2 \left( 
A_t \frac{L\gamma^2}{2} - A_t  \gamma  + \tilde A_t \tilde \gamma^2
+ \left((\delta + \tilde \delta)^2-\delta \right)\tilde B_t \right)\label{l1}\\
& + \Vert \tilde z -  z^*\Vert^2_{\mathbf{\Lambda}^+} \left(-\tilde \alpha \tilde C_t + \tilde C'_t \right)\label{l2}\\
& + \Vert \tilde y - y^* \Vert^2 ( \tilde B'_t   -\frac \delta 2 \tilde B_t)\label{l3}\\
&+\Vert x-x^*\Vert^2(\tilde A_t \tilde \eta - \tilde A_t \frac{\nu \tilde \gamma}{2}) \label{l4}\\
& +\Vert \tilde x - x^* \Vert^2 (\tilde A'_t - \tilde A_t \tilde \eta) \label{l5}\\
& + \Vert \nabla F(x) -\nabla F(x^*) \Vert^2 (\delta \tilde B_t - \frac{\tilde \gamma}{2L}\tilde A_t )\label{l6}\\
& +\sum_{(i,j)\in \mathcal{E} }\lambda_{ij} \Vert (\mathbf{e}_i-\mathbf{e}_j)(\mathbf{e}_i-\mathbf{e}_j)^\mathsf{T}(y+z) \Vert^2\left( \chi_2\tilde\beta^2\tilde C_t+\beta(\beta-1)C_t\right) \label{l8} \\
& + \Vert z - z^*\Vert^2  ( \chi_1\tilde \alpha -\tilde \tau\frac 12 \tilde \beta \frac{\nu}L) \tilde C_t\label{l9}\\
& + \Vert y - y^* \Vert (B'_t -(1-\tilde \tau \frac{\nu}L \frac{2\theta}{\delta})\alpha B_t )\label{l10}\\
& + \Vert y + z - y^* -z^* \Vert^2 ( C'_t - \alpha C_t) \label{l11} \\
& +d_F(x,x^*)( A'_t -\eta A_t)\label{l12}\\
& + d_F(\tilde x,x)( -A_t\eta  +2\tilde\gamma \tilde A_t) \label{l13}\\
& +d_F(\tilde x,x^*) ( A_t\eta -2\tilde\gamma \tilde A_t) \label{l14}\\
& +\langle \nabla F(x)-\tilde y,y^*-\tilde y\rangle ( -2 \tilde \delta \tilde B_t + \gamma A_t) \label{l15}\\
&+\langle \tilde y-y^*,\tilde x-x^*\rangle \left( 2\tilde\gamma \tilde A_t -2\theta \nu \tilde B_t\right)\label{l16}
\end{align}

\paragraph{Resolution}

\label{ResoGD}
\begin{proof}[Proof of Lemma \ref{important-lemma}] Our goal is to put to zero all of the terms appearing next to scalar products and make the factors of positive quantities (norms or divergences) less or equal to zero. Given our relations, we guess that each exponential has the same rate. Thus, with $\tau>0$, we fix $\frac{\delta}2=\tilde \eta=\eta=\tilde\alpha=\tau\sqrt{\frac \nu L}$, which leads to $\tilde \gamma =\frac{2\tau}{\sqrt{\nu L}}$ using Eq. \eqref{l4}. Also, from Eq. \eqref{l14}:
$$4\tilde A_t=\nu A_t.$$
Next, from Eq. \eqref{l6} and Eq. \eqref{l16}, it's necessary that: $$2L\delta=\theta \nu\,,$$
thus $\theta=4\tau\sqrt{\frac L{\nu}}$. From Eq. \eqref{l16}, we get:
$$\tilde A_t=2L\nu\tilde B_t.\,$$
Combining this previous equation with Eq. \eqref{l15}, as $4\tilde A_t=\nu A_t$, we have $\tilde \delta=4L\gamma $. 
Next, Eq. \eqref{l1} gives, with the equations above:
\begin{align*}
A_t(\frac{L\gamma^2}{2} - \gamma)  + \tilde A_t \tilde \gamma^2
+ \left((\delta + \tilde \delta)^2-\delta \right) \tilde B_t
&=A_t \frac{L\gamma^2}{2} - A_t  \gamma  + \frac{\nu}{4}\tilde \gamma^2A_t + \left(\delta^2 + \tilde \delta^2 +\delta \right) \frac{A_t}{8L}\\
&=A_t \left(\frac{L\gamma^2}{2} -  \gamma + \frac{\nu}{4}\frac{4\tau^2}{\nu L}\right) +A_t(2\tau \sqrt{\frac{\nu}{L}} +4\tau^2\frac{\nu}{L} + 16 L^2\gamma^2) \frac{1}{8L} \notag \\
& \leq A_t(\gamma^2\frac 52L - \gamma + \frac 5 4\frac {\tau^2}L + \frac{\sqrt{2}}{8} \frac \tau L)
\end{align*}
We thus pick $\gamma=\frac 1{ 4 L}$ and $\tau= \frac{1}{8}$, so that $\tilde \delta=1$. Via Eq. \eqref{l10}, we fix $\tilde \tau=\frac 18<1$. With Eq. \eqref{l9}, we then get:
$$\tilde \beta=2 \chi_1\sqrt{\frac L\nu}$$

We also put $\alpha=2\tau \sqrt{\frac \nu L}$ and only one last equation, Eq. \eqref{l8}, needs to be satisfied, for which we pick $\beta=\frac 12$:
$$ \chi_2\tilde \beta^2\tilde C_t+\beta(\beta-1)C_t=( \chi_2\tilde \beta \alpha-\frac 14)C_t$$
This implies that $ \chi_2 \chi_1\leq \frac 12$. In summary, we set:
$$\begin{array}{lllllll}
        \eta  = \frac 18 \sqrt{\frac \nu L}
         &  \; \gamma = \frac{1}{4L}
         &  \;\delta = \frac 14 \sqrt{\frac \nu L}
         & \;\alpha = \frac 14 \sqrt{\frac \nu L}
         & \;\beta = \frac 12
         &  \;\theta = \frac 12 \sqrt{\frac L \nu }&\; \tilde\tau=\frac 18
         \\
         \tilde \eta = \frac 18 \sqrt{\frac \nu L}
         & \; \tilde \gamma = \frac 1{4 \sqrt{\nu L}}
         & \; \tilde \delta = 1
         & \;\tilde \alpha = \frac 18 \sqrt{\frac \nu L}
         & \;\tilde \beta = 2  \chi_1\sqrt{\frac L \nu }
         & \; \nu = \frac \mu 2 &\; \tau=\frac 18
    \end{array}$$

Now, let's pick: $A_t=e^{t \tau \sqrt{\frac{\nu}L}}=e^{\frac{t}{8\sqrt 2}\sqrt{\frac{\mu}L}}$ so that:

$$\begin{array}{lllllll}
A_t=e^{\frac{t}{8\sqrt 2}\sqrt{\frac{\mu}L}} & \; \tilde A_t=\frac{\mu}8e^{\frac{t}{8\sqrt 2}\sqrt{\frac{\mu}L}} & \; \tilde B_t=\frac{1}{8L}e^{\frac{t}{8\sqrt 2}\sqrt{\frac{\mu}L}} \\
B_t=\frac{1}{16L}e^{\frac{t}{8\sqrt 2}\sqrt{\frac{\mu}L}} & \; C_t=\frac{1}{2\mu}e^{\frac{t}{8\sqrt 2}\sqrt{\frac{\mu}L}}  & \; \tilde C_t=\frac{1}{32 \chi_1L}e^{\frac{t}{8\sqrt 2}\sqrt{\frac{\mu}L}}\,.
\end{array}$$

This implies that $\Psi\leq 0$.

\end{proof}

\section{Comparison of complexities with related work}
\label{proof-optimality-dadao}
\begin{proof}[Proof of Prop. \ref{optimality-dadao}]
We consider the settings of concurrent works and, given any gossip matrix admissible for them, we show that DADAO has better rates.

\paragraph{Comparison with ADOM+ \cite{kovalev2021lower}.} The ADOM+ setting is:
\begin{itemize}
\item \textit{Gossip matrices}: Laplacians $\mathcal L$ with $\Vert \mathcal{L} x-x\Vert^2\leq (1-\frac 1{\chi})\Vert x\Vert^2$ for $\chi \geq 1$ and $x \in \mathbf{1}^\perp$.
\item \textit{Total number of gradients to reach $\epsilon$ precision}: $\mathcal{O}(n\sqrt{\frac L\mu}\log \frac{1}{\epsilon})$.
\item \textit{Total number of edges activated to reach $\epsilon$ precision}: $\mathcal{O}(|\mathcal{E}|\chi\sqrt{\frac L\mu}\log \frac{1}{\epsilon})$.
\end{itemize}
If we take an eigenvector of $\mathcal L$ for the eigenvalue $\frac 1 {\chi_1[\mathcal L]}$, then the Laplacian inequality directly leads to $\left (1 - \frac 1 {\chi_1[\mathcal L]} \right)^2 \leq 1-\frac 1{\chi}$ and we have:
\begin{align*}
    \frac 1 {\chi_1[\mathcal L]} \left ( \frac 1 {\chi_1[\mathcal L]} - 2 \right) \leq - \frac 1 \chi \, ,
\end{align*}
leading to:
\begin{align*}
    \frac 2 {\chi_1[\mathcal L]} \geq \frac 1 {\chi_1[\mathcal L]} \left ( 2 - \frac 1 {\chi_1[\mathcal L]} \right) \geq \frac 1 \chi \, .
\end{align*}
Thus, $ \chi_1[\mathcal{L}]\leq 2 \chi$. Remind that, by definition, $ \chi_2 \leq  \chi_1$. Note that in ADOM+, we also have $\Vert \mathcal L \Vert \leq 2$, so that $\mathrm{Tr} (\mathcal L) \leq 2n$. Then, for any Laplacian matrix $\mathcal L$ valid for ADOM+, we consider for DADAO a $ \Lambda  $ defined as followed: $$\Lambda =\sqrt{2 \chi_1[\mathcal{L}] \chi_2[\mathcal{L}]}\mathcal{L}$$
Then, DADAO has the same gradient complexity as ADOM+, but the number of communications of DADAO is:
$$\frac T2\mathrm{Tr}\,\Lambda \leq \frac 12\sqrt{2 \chi_1[\mathcal{L}] \chi_2[\mathcal{L}]}\sqrt{\frac L\mu}\log \left( \frac{1}{\epsilon} \right)2n = \mathcal{O}(|\mathcal{E}|\chi\sqrt{\frac L\mu}\log \frac{1}{\epsilon})$$

Consequently, DADAO is better than ADOM+ for all valid configurations of ADOM+ (in the fixed graph topology setting) and DADAO. We note that for the complete graph, $\chi= \mathcal O (1)$ and $\vert \mathcal E \vert= \mathcal O (n^2)$, whereas $n \sqrt{2 \chi_1[\mathcal{L}] \chi_2[\mathcal{L}]} = \mathcal O (n)$ and DADAO has strictly better communication rates than ADOM+ on this graph.
\paragraph{Comparison with Gradient Tracking methods AGT, OGT \cite{li2021accelerated, song2021optimal}.} The setting is described by:
\begin{itemize}
\item \textit{Gossip matrices}: Any matrix $\mathcal L = \mathbf{I} - W$ with $W$ symmetric doubly-stochastic, \ie~such that $\sum_iW_{ij}=1$ and $\sum_j W_{ij}=1$.
\item \textit{Total number of gradients to reach $\epsilon$ precision}: $\mathcal{O}(n\sqrt{\frac L\mu}\log \frac{1}{\epsilon})$.
\item \textit{Total number of edges activated to reach $\epsilon$ precision}: $\mathcal{O}(|\mathcal{E}|\frac 1{\sqrt{\theta}}\sqrt{\frac L\mu}\log \frac{1}{\epsilon})$, where $\theta=1-\Vert W-\frac 1 n \mathbf{1}\mathbf{1}^{\mathsf{T}}\Vert$.
\end{itemize}
If $\kappa =\mathcal{O}(\frac{|\mathcal{E}|}{n})$, using Lemma \ref{chi-lemma}, we have:

$$\frac 12 \sqrt{2\chi_1[\mathcal L]\chi_2[\mathcal L]}\mathrm{Tr}\,\mathcal{L}= \mathcal{O}(|\mathcal{E}|\sqrt{\rho[\mathcal L]})$$

Furthermore, as $W$ is stochastic, $\Vert \mathcal L \Vert \leq 2$ and $\theta \leq \frac 1 {\chi_1[\mathcal L]}$, leading to $\rho \leq \frac 2 \theta$. Consequently, the communication complexity of DADAO run for $\mathcal{O}(\sqrt{\frac L\mu}\log \frac{1}{\epsilon})$ iterations recovers the rate of GT. Furthermore, for the complete graph, for any Laplacian $\mathcal L$ with $ \rho[\mathcal L] =  \mathcal O (1)$ (remind that, by definition $\rho \geq 1$), we have $\sqrt{2\chi_1[\mathcal L]\chi_2[\mathcal L]}\mathrm{Tr}\,\mathcal{L} = \mathcal O (n)$ whereas $\vert \mathcal E \vert = \mathcal O (n^2)$, thus DADAO uses an order of magnitude less communications than GT need to. 

Note that for the Star-graph, there is no admissible Laplacian in the framework of \cite{song2021optimal}.
\paragraph{Comparison with MSDA \cite{scaman17a}.} The setting is described by:
\begin{itemize}
\item \textit{Gossip matrices}: any Laplacians admissible for DADAO.
\item \textit{Total number of gradients to reach $\epsilon$ precision}: $\mathcal{O}(n\sqrt{\frac L\mu}\log \frac{1}{\epsilon})$.
\item \textit{Total number of edges activated to reach $\epsilon$ precision}: $\mathcal{O}(|\mathcal{E}|\sqrt{\rho}\sqrt{\frac L\mu}\log \frac{1}{\epsilon})$.
\end{itemize}
If $\kappa =\mathcal{O}(\frac{|\mathcal{E}|}{n})$, using Lemma \ref{chi-lemma}, we see that, for any Laplacian $\mathcal L$:

$$\frac 12 \sqrt{2\chi_1[\mathcal L]\chi_2[\mathcal L]}\mathrm{Tr}\,\mathcal{L}= \mathcal{O}(|\mathcal{E}|\sqrt{\rho})$$

Consequently, the communication complexity of DADAO run for $\mathcal{O}(\sqrt{\frac L\mu}\log \frac{1}{\epsilon})$ iterations is better than MSDA. Furthermore, for the complete graph, for any Laplacian $\mathcal L$ with $ \rho[\mathcal L] =  \mathcal O (1)$ (remind that, by definition $\rho \geq 1$), we have $\sqrt{2\chi_1[\mathcal L]\chi_2[\mathcal L]}\mathrm{Tr}\,\mathcal{L} = \mathcal O (n)$ whereas $\vert \mathcal E \vert = \mathcal O (n^2)$, thus DADAO uses an order of magnitude less communications than MSDA need to.

\paragraph{Comparison with the Continuized framework \cite{even2021continuized}.} 
In this framework and using their notations, we have:
\begin{itemize}
\item \textit{Gossip matrices}: all Laplacian matrices $\mathcal L$ verifying $\mathrm{Tr} \mathcal L = 2$.
\item \textit{Total number of gradients to reach $\epsilon$ precision}: $\mathcal{O}(\frac{1}{\theta'_\text{ARG}}\sqrt{\frac L\mu}\log \frac{1}{\epsilon})$.
\item \textit{Total number of edges activated to reach $\epsilon$ precision}: $\mathcal{O}(\frac{1}{\theta'_\text{ARG}}\sqrt{\frac L\mu}\log \frac{1}{\epsilon})$.
\end{itemize}
First, we slightly rephrase one of the proposition of \cite{even2021continuized} to match our notations:

\begin{lemma}\label{continuized}For a Laplacian $\mathcal L$ with $\mathrm{Tr}\,\mathcal{L}=2$, the communication and gradient complexity of the continuized framework is given by:
    $$\mathcal{O}(\sqrt{\chi_1 [\mathcal L]\chi_2[\mathcal L]}\sqrt{\frac L \mu}\log \frac 1\epsilon)\,.$$
\end{lemma}
\begin{proof}
Under the notation of \cite{even2021continuized}, $\theta'_{\text{ARG}} = \sqrt{\mu_\text{gossip}/\max_{\{v,w\}} \frac{R_{vw}}{P_{vw}}}$, with $\mu_\text{gossip} = \frac 1 {\chi_1[\mathcal L]}$ in our setting. Moreover, we remind that in \cite{even2021continuized}, $\mathcal{L}=AA^\mathsf{T}$ and that $Ae_{vw}=\sqrt{P_{vw}}(e_v-e_w)$. Thus, by definition:
\begin{align}
\frac{R_{vw}}{P_{vw}}&\triangleq \frac{e_{vw}^\mathsf{T}A^+Ae_{vw}}{P_{vw}}\\
&=\frac{e_{vw}^\mathsf{T}A^+(e_v-e_w)}{\sqrt{P_{vw}}}\\
&=\frac{(A^{+\mathsf{T}}e_{vw})^\mathsf{T}(e_v-e_w)}{\sqrt{P_{vw}}}\\
&=\frac{((AA^\mathsf{T})^{+\mathsf{T}}Ae_{vw})^\mathsf{T}(e_v-e_w)}{\sqrt{P_{vw}}}\\
&=(e_v-e_w)^\mathsf{T}\mathcal{L}^+(e_v-e_w)\, ,
\end{align}
and we have $\theta'_{\text{ARG}} = \frac 1 {\sqrt{ 2\chi_1[\mathcal L] \chi_2[\mathcal L]}}$.
\end{proof}
Next, if $\mathrm{Tr}\,\mathcal L=2$, we proved in Lemma \ref{chi-lemma}that $2\sqrt{\chi_1[\mathcal L]\chi_2[\mathcal L]}\geq (n-1)$. Thus, we see that the gradient complexity of DADAO in $\mathcal{O}(n\sqrt{\frac L\mu}\log \frac{1}{\epsilon})$ is always better than  the one of the continuized framework. If we write $f(n) = \Omega( g(n))$ the fact that there is a constant $C>0$ and $n_0 \in \mathbb{N}$ such that $\forall n\geq n_0, \; C g(n) \leq f(n)$, then, in the cycle graph, for any $\mathcal L$ with $\mathrm{Tr}\,\mathcal L=2$, $\chi_1[\mathcal L] = \Omega(n^3)$ and $\chi_2[\mathcal L] = \Omega(n)$. Thus DADAO uses an order of magnitude less gradients than the continuized framework for this graph.

\end{proof}

\section{Practical Implementation}
\label{PracticalImplem}

In this section, we describe in more detail the implementation of our algorithm. As we did not physically execute our method on a compute network but carried it out on a single machine, all the asynchronous computations and communications had to be simulated. Thus, we will first discuss the method we followed to simulate our asynchronous framework before detailing the practical steps of our algorithm through a pseudo-code.

\subsection{Simulating the Poisson Point Processes}
\label{simuPPP}
To emulate the asynchronous setting, before running our algorithm, we generate $2$ independent sequences of jump times at the graph's scale: one for the computations and one for the communications. As we considered independent P.P.Ps, the time increments follow a Poisson distribution. At the graph's scale, each node spiking at a rate of $1$, the Poisson parameter for the gradient steps process is $n$. Following the experimental setting of the Continuized framework \citep{even2021continuized}, we considered that all edges in $ \mathcal E  $ had the same probability of spiking. Thus, given any graph $\mathcal E  $ and $\mathcal L$ its corresponding standard Laplacian with unit edge weights, we computed the parameter $\lambda$ of the communication process:
\begin{equation}
    \lambda = \sqrt{2 \chi_1\left[\frac{\mathcal L }{\vert \mathcal E   \vert}\right] \chi_2\left[\frac{\mathcal L }{\vert \mathcal E   \vert}\right]}.
\end{equation}
Having generated the $2$ sequences of spiking times at the graph's scale, we run our algorithm playing the events in order of appearance, attributing the \textit{location} of the events by sampling uniformly one node if the event is a gradient step or sampling uniformly an edge in $\mathcal{E}  $ if it is a communication. \\
\subsection{Pseudo Code}\label{parameters}
We keep the notations introduced in Eq. \eqref{dynamic} and recall the following constant values specified in Appendix \ref{ResoGD}:
\begin{equation*}
    \begin{array}{llllll}
        \eta  = \frac 18 \sqrt{\frac \nu L}
         &  \; \gamma = \frac{1}{4L}
         &  \;\delta = \frac 14 \sqrt{\frac \nu L}
         & \;\alpha = \frac 14 \sqrt{\frac \nu L}
         & \;\beta = \frac 12
         &  \;\theta = \frac 12 \sqrt{\frac L \nu }
         \\
         \tilde \eta = \frac 18 \sqrt{\frac \nu L}
         & \; \tilde \gamma = \frac 1{4 \sqrt{\nu L}}
         & \; \tilde \delta = 1
         & \;\tilde \alpha = \frac 18 \sqrt{\frac \nu L}
         & \;\tilde \beta = 2  \chi_1[\Lambda] \sqrt{\frac L \nu }
         & \; \nu = \frac \mu 2
    \end{array}
\end{equation*}
For the sake of completeness, we also specify the matrix $\mathcal A$ describing the linear ODE \eqref{ODE}:
\begin{equation}
\label{def-A}
    \mathcal A = 
    \begin{pmatrix}
    -\eta & \eta & 0 & 0 & 0 & 0 \\
    \tilde \eta & -\tilde \eta & 0 & 0 & 0 & 0 \\
    0 & 0 & -\alpha & \alpha & 0 & 0 \\
    0 & - \theta \nu & - \theta & 0 & -\theta & 0 \\
    0 & 0 & 0 & 0 & - \alpha & \alpha \\
    0 & 0 & 0 & 0 & \tilde \alpha & - \tilde \alpha 
    \end{pmatrix} 
\end{equation}
As described in Appendix \ref{simuPPP}, we call \texttt{PPPspikes} the process mentioned above, returning the ordered sequence of events and time of spikes of the two P.P.Ps. Then, we can write the pseudo-code of our implementation of the DADAO optimizer in Algorithm \ref{PseudoCode}.

 \begin{center}
\resizebox{0.95\columnwidth}{!}{
\begin{algorithm}[H]
\caption{Pseudo-code of our implementation of DADAO on a single machine.}
\SetAlgoLined
\SetKwBlock{Init}{Initialize}{}
\SetKwProg{Para}{In parallel}{ continuously do:}{}
\SetKwComment{Comment}{// }{}
\label{PseudoCode}

\KwIn{On each machine $i \in \{1,...,n\}$, an oracle able to evaluate $\nabla f_i$, Parameters $\mu, L,  \chi_1, t_{\max}, n, \lambda$. \\
$\quad \quad \quad $ The graph $\mathcal E $.}
\Init(on each machine $i \in \{1,...,n\}$:){
Set $X^{(i)} = (x_i, \tilde x_i, \tilde y_i)$ and $Y^{(i)} = (y_i, z_i, \tilde z_i)$ to $0$ \;
Set constants $\nu,\tilde\eta,\eta,\gamma,\alpha,\tilde\alpha,\theta,\delta,\tilde\delta,\beta,\tilde\beta$ using $\mu, L,  \chi_1$\;
Set $\mathcal A$\;
$T^{(i)} \gets 0$ \;
}
\texttt{ListEvents}, \texttt{ListTimes} $\gets$ \texttt{PPPspikes}$(n, \lambda, t_{\max})$ \;
$n_\text{events} \gets \vert \texttt{ListEvents}\vert$ \;
\For{$k \in [\![1, n_\textup{\text{events}}]\!]$}{

\If{\textup{\texttt{ListEvents}}$[k]$ is to take a gradient step}{
$i \sim \mathcal U ([\![1,n]\!])$ \;
$\begin{pmatrix}
X^{(i)} \\
Y^{(i)}
\end{pmatrix}
 \gets \exp \left( (\textup{\texttt{ListTimes}}[k]-T^{(i)})\mathcal{A}\right)
\begin{pmatrix}
X^{(i)} \\
Y^{(i)}
\end{pmatrix}$ \;
$g_i \gets \left( \nabla f_i(x_i) - \nu x_i - \tilde y_i \right)$\tcp*[r]{Local gradient computation.}
$x_i \gets x_i -\gamma g_i$ \;
$\tilde x_i \gets \tilde x_i - \tilde \gamma g_i$ \;
$\tilde y_i \gets \tilde y_i + (\delta + \tilde \delta) g_i$ \;
$T^{(i)}\gets \texttt{ListTimes}[k]$ \;
}
\ElseIf{\textup{\texttt{ListEvents}}$[k]$ is to take a communication step}{
$(i,j) \sim \mathcal U (\mathcal E)$ \;
$\begin{pmatrix}
X^{(i)} \\
Y^{(i)}
\end{pmatrix}
 \gets \exp \left( (\textup{\texttt{ListTimes}}[k]-T^{(i)})\mathcal{A}\right)
\begin{pmatrix}
X^{(i)} \\
Y^{(i)}
\end{pmatrix}$ \;
$\begin{pmatrix}
X^{(j)} \\
Y^{(j)}
\end{pmatrix}
 \gets \exp \left( (\textup{\texttt{ListTimes}}[k]-T^{(j)})\mathcal{A}\right)
\begin{pmatrix}
X^{(j)} \\
Y^{(j)}
\end{pmatrix}$ \;
$m_{ij} \gets (y_i + z_i - y_j - z_j)$\tcp*[r]{Message exchanged.} 
$z_i \gets z_i - \beta m_{ij}$\;
$\tilde z_i \gets \tilde z_i -\tilde \beta m_{ij}$\;
$z_j \gets z_j + \beta m_{ij}$\;
$\tilde z_j \gets \tilde z_j + \tilde \beta m_{ij}$\;
$T^{(i)}\gets \texttt{ListTimes}[k]$\;
$T^{(j)}\gets \texttt{ListTimes}[k]$\;
}

}
\Return{$(x_i)_{1 \leq i \leq n}$, the estimate of $x^*$ on each worker $i$.}
\end{algorithm}
}
\end{center}


\end{document}